\documentclass[12pt]{amsart}

\usepackage[T1]{fontenc}
\usepackage{amsmath,amssymb}
\usepackage{amscd}
\usepackage{pb-diagram}
\newtheorem{prop}{Theorem}[subsection]
\newtheorem{lemma}[prop]{Lemma}
\newtheorem{corollary}[prop]{Corollary}
\newtheorem{definition}[prop]{Definition}

\newtheorem{remark}[prop]{Remark}
\newcommand{\cd}{\cdot}
\newcommand{\clc}{\cdot\ldots\cdot}
\newcommand{\ot}{\otimes}

\newcommand{\op}{\oplus}
\newcommand{\bop}{\bigoplus}
\newcommand{\oplop}{\oplus \ldots \oplus}
\newcommand{\olo}{\otimes\ldots\otimes}
\newcommand{\plp}{+ \ldots +}
\newcommand{\we}{\wedge}
\newcommand{\wlw}{\wedge\ldots\wedge}

\newcommand{\ci}{\circ}

\newcommand{\ti}{\times}
\newcommand{\nn}{\mathbb{N}}
\newcommand{\zz}{\mathbb{Z}}

\newcommand{\cc}{\mathbb{C}}

\newcommand{\ff}{\mathbb{F}}
\newcommand{\al}{\alpha}

\newcommand{\ga}{\gamma}
\newcommand{\de}{\delta}

\newcommand{\ep}{\varepsilon}

\newcommand{\io}{\iota}

\newcommand{\la}{\lambda}
\newcommand{\La}{\Lambda}
\newcommand{\si}{\sigma}
\newcommand{\Si}{\Sigma}
\newcommand{\te}{\theta}

\newcommand{\da}{\dagger}

\newcommand{\co}{\equiv}
\newcommand{\C}[1]{\mathcal{#1}}

\newcommand{\T}[1]{\textrm{#1}}
\newcommand{\E}[1]{\emph{#1}}
\newcommand{\B}[1]{\mathbb{#1}}

\newcommand{\fork}[2]{\left\{ \begin{array}{#1} #2 \end{array} \right.} 
\newcommand{\arr}[2]{\begin{array}{#1} #2 \end{array}}
\newcommand{\mat}[2]{\left(\begin{array}{#1} #2 \end{array} \right)}
\newcommand{\arle}{\dgARROWLENGTH=0.5\dgARROWLENGTH}
\newcommand{\no}[1]{\node{#1}}
\newcommand{\ar}[2]{\arrow{#1}{#2}}

\newcommand{\su}{\subseteq}
\newcommand{\q}{\qquad}
\newcommand{\qq}{\qquad \qquad}

\newcommand{\wih}{\widehat}

\newcommand{\deht}[2]{\T{det}\big(H_{#1}(#2)\big)}
\newcommand{\dehe}[2]{\E{det}\big(H_{#1}(#2)\big)}
\newcommand{\dht}[2]{\T{dim}\big(H_{#1}(#2) \big)}
\newcommand{\dhe}[2]{\E{dim}\big(H_{#1}(#2) \big)}

\numberwithin{equation}{section}

\begin{document}
 \title{Joint torsion of several commuting operators}

\author{J. Kaad}
\thanks{\\ 2010 \emph{Mathematical Subject Classification: Primary:
    $47$A$13$, Secondary: $15$A$15$, $18$G$35$, $19$C$20$, $47$B$13$} \\
\emph{Keywords and phrases: Determinant, Koszul homology, multiplicative
  Fredholm theory, secondary invariants.} \\
  \\
  This work was supported by the Australian Research Council.}
\maketitle
\vspace{-10pt}
\centerline{ Hausdorff Center for Mathematics, Bonn University}
\centerline{ Endenicher Allee $60$, $53115$ Bonn}
\centerline{ email: jenskaad@hotmail.com}
\bigskip
\vspace{30pt} 

\centerline{\textbf{Abstract}}
We introduce the notion of joint torsion for several commuting operators
satisfying a Fredholm condition. This new secondary invariant takes values in
the group of invertibles of a field. It is constructed by comparing
determinants associated with different filtrations of a Koszul complex. Our
notion of joint torsion generalizes the Carey-Pincus joint torsion of
a pair of commuting Fredholm operators. As an example, under more
restrictive invertibility assumptions, we show that the joint torsion
recovers the multiplicative Lefschetz numbers. Furthermore, in the
case of Toeplitz operators over the polydisc we provide a link between
the joint torsion and the Cauchy integral formula. We will also
consider the algebraic properties of the joint torsion. They include a
cocycle property, a symmetry property, a triviality property and a
multiplicativity property. The proof of these results relies on a
quite general comparison theorem for vertical and horizontal torsion
isomorphisms associated with certain diagrams of chain complexes.

\newpage
\tableofcontents
\newpage

\section{Introduction}
Let us start by presenting a short overview of some results of Richard Carey
and Joel Pincus related to their joint torsion invariant. This should serve as
sufficient motivation for the introduction of the multivariable generalization
which we will discuss afterwards.

The point of depart is a pair of commuting Fredholm operators $(A,B)$ on some
vector space $E$ over a field $\ff$. Thus, we assume that all kernels and
cokernels are finite dimensional and that the commutator $[A,B] = 0$ is
trivial. The kernel and cokernel of $A$ then fit in a long exact sequence of
vector spaces
\begin{equation}\label{eq:akosz}
\C E^A : \begin{diagram}
\arle
\no{0} \ar{e,t}{} \no{\T{Ker}(A) \cap \T{Ker}(B)} \ar{e,t}{} \no{\T{Ker}(A)}
\ar{e,t}{-B} \no{\T{Ker}(A)} \ar{e,t}{}  \no{H_1(A,B)} \ar{s,r}{} \\
\no{} \no{0} \no{E/\big(\T{Im}(A) + \T{Im}(B)\big)} \ar{w,b}{}
\no{\T{Coker}(A)} \ar{w,b}{} \no{\T{Coker}(A)} \ar{w,b}{B}
\end{diagram}
\end{equation}
Here $H_1(A,B)$ is the first Koszul homology group of the commuting pair
$(A,B)$. Likewise, the kernel and cokernel of $B$ fit in another long exact
sequence of vector spaces
\begin{equation}\label{eq:bkosz}
\C E^B : \begin{diagram}
\arle
\no{0} \ar{e,t}{} \no{\T{Ker}(A) \cap \T{Ker}(B)} \ar{e,t}{} \no{\T{Ker}(B)}
\ar{e,t}{A} \no{\T{Ker}(B)} \ar{e,t}{}  \no{H_1(A,B)} \ar{s,r}{} \\
\no{} \no{0} \no{E/\big(\T{Im}(A) + \T{Im}(B)\big)} \ar{w,b}{}
\no{\T{Coker}(B)} \ar{w,b}{} \no{\T{Coker}(B)} \ar{w,b}{-A}
\end{diagram}
\end{equation}
We can then apply a determinant functor to these two long exact sequences of
finite dimensional vector spaces. After some canonical identifications, we
obtain two possibly distinct isomorphisms
\[
\begin{split}
T(A) \, \T{ and }\, T(B) & : 
\T{det}\big( \T{Ker}(A) \cap \T{Ker}(B)\big) \ot \T{det}\big( E/\big(\T{Im}(A)
+ \T{Im}(B) \big) \big) \\
& \qq \longrightarrow \T{det}\big( H_1(A,B) \big)
\end{split}
\]
at the level of determinants. The quotient of these two isomorphisms is an
automorphism of a one-dimensional vector space and can thus be identified with
an invertible number. This number is, up to a sign, the \emph{Carey-Pincus
  joint torsion} of the commuting pair of Fredholm operators,
\[
\tau(A,B) = (-1)^{\nu(A,B)} T(A)^{-1} \ci T(B) \in \ff^*
\]
Here the exponent $\nu(A,B) \in \nn \cup \{0\}$ is given by dimensions of
Koszul homology groups. We refer to the paper \cite{carpincI} for more details
on this construction.

Let us pass to a description of some important results. First of all, we
outline the relation to the second algebraic $K$-group. This relation can be
explained by means of the determinant invariant of Larry Brown. The
determinant invariant is a homomorphism
\[
d : K_2(\C L/ \C F) \to \ff^*
\]
from the second algebraic $K$-group of the quotient ring $\C L/ \C F$ to the
group of invertibles of the field. Here $\C L := \C L(E)$ is the linear
operators on the vector space $E$ and $\C F := \C F(E)$ is the ideal of finite
rank operators. For more details on the determinant invariant we refer to the
book of Jonathan Rosenberg \cite{rosenI} and the papers
\cite{brown,kaad,rosenII}. The two commuting Fredholm operators $A$ and $B$
determine two invertible and commuting elements in the quotient ring $\C L/ \C
F$. In particular, we get a Steinberg symbol
\[
\{ q(A), q(B)\} \in K_2(\C L/\C F)
\]
in the second algebraic $K$-group of the quotient ring. Here $q : \C L \to \C
L/\C F$ denotes the quotient map. The relation between joint torsion and
algebraic $K$-theory can now be stated.

\begin{prop}\cite[Theorem $2$]{carpincI}\label{jointstein}
The joint torsion of the commuting pair of Fredholm operators $(A,B)$ agrees
with the determinant invariant of the Steinberg symbol $\{q(A),q(B)\} \in
K_2(\C L/\C F)$. Thus, in formulas we have the identity
\[
\tau(A,B) = d\{q(A),q(B)\}
\]
in the group of invertibles $\ff^*$.
\end{prop}

As an important consequence of the above description we get a multiplicativity
property for the joint torsion. Indeed, since the Steinberg symbol is bilinear
we get that
\begin{equation}\label{eq:symbjoint}
\tau(A \cd C,B) = \tau(A,B) \cd \tau(C,B)
\end{equation}
Here $C$ is an extra Fredholm operator which commutes with $B$. It also
follows from Theorem \ref{jointstein} that the joint torsion is invariant
under finite rank perturbations.

The next result, which we would like to describe, is a complete calculation of
the joint torsion in the case of Toeplitz operators over the disc. This can be
thought of as a multiplicative index theorem. We look at two continuous and
invertible functions on the circle
\[
f,g : S^1 \to \cc^*
\]
We will then assume that these two functions extend to holomorphic functions
on the disc. The associated Toeplitz operators on Hardy space
\[
\arr{ccc}{
T_f \, \T{ and }\, T_g : H^2(\B D^\ci) \to H^2(\B D^\ci)
}
\]
then form a pair of commuting Fredholm operators, $(T_f,T_g)$. In particular,
we can assign a joint torsion
\[
\tau(T_f,T_g) \in \cc^*
\]
to the pair of Toeplitz operators.

On the other hand, we could look at the behaviour of the meromorphic function
$f/g$ near the zeros and poles. To be precise, the holomorphic functions $f$
and $g$ have a finite number of zeros $\la_1,\ldots,\la_n \in \B D^\ci$ in the
interior of the disc. The multiplicities of a zero $\la_i \in \B D^\ci$ will
be denoted by 
\[
\mu_f(\la_i) \,\T{ and }\, \mu_g(\la_i) \in \nn \cup \{0\}
\]
for the functions $f$ and $g$ respectively. The following limit of quotients
\[
c_{\la_i}(f,g) = (-1)^{\mu_f(\la_i) \cd \mu_g(\la_i)}
\lim_{z \to \la_i} \frac{f(z)^{\mu_g(\la_i)}}{g(z)^{\mu_f(\la_i)}} \in \cc^*
\]
is then a well-defined invertible number. The product of these limits
\[
c(f,g) = \prod_{i=1}^n c_{\la_i}(f,g) \in \cc^*
\]
is known as the \emph{tame symbol} of $f$ and $g$. For more details we refer
to the paper \cite{deligne} by Pierre Deligne.

\begin{prop}\cite[Proposition $1$]{carpincI}\label{tametame}
The joint torsion of the pair of Toeplitz operators $(T_f,T_g)$ coincides with
the tame symbol of the pair of functions $(f,g)$. Thus, in formulas we have
the identity
\[
\tau(T_f,T_g) = c(f,g)
\]
\end{prop}

Notice that there are other descriptions available for the tame symbol. For
example, this quantity can be expressed as the monodromy of a flat line
bundle. See \cite{beilinson, carpincI, deligne}.

The very simple question which we will investigate in the present paper can
now be formulated:

\emph{"What happens when we replace a pair of commuting Fredholm operators by
  a commuting tuple?"}

Thus, let us consider a commuting tuple $A = (A_1,\ldots,A_n)$ of linear
operators on the vector space $E$. This commuting tuple gives rise to a Koszul
complex $K(A)$. We will think of the Koszul complex as a $\zz_2$-graded chain
complex and denote the $\zz_2$-graded homology group by $H(A) = H_+(A) \op
H_-(A)$. Our first task is to find a replacement for the long exact sequences
\eqref{eq:akosz} and \eqref{eq:bkosz}. One possibility is to remove one of the
operators from the commuting tuple $A$, say the operator $A_i$. Thus, we could
look at the commuting tuple
\[
i(A) = (A_1,\ldots,\wih{A_i},\ldots,A_n)
\]
for some $i \in \{1,\ldots,n\}$. It can then be proved that we have a short
exact sequence of Koszul complexes
\begin{equation}\label{eq:shkosz}
\begin{CD}
0 @>>> K(i(A)) @>>> K(A) @>>> K(i(A))[1] @>>> 0
\end{CD}
\end{equation}
Here the notation "$[1]$" refers to the operation of changing both the sign of
the differential and the grading of a $\zz_2$-graded chain complex. In
particular, we get a six term exact sequence of even and odd Koszul homology
groups,
\begin{equation}\label{eq:sixiho}
\begin{CD}
H_+(i(A)) @>>> H_+(A) @>>> H_-(i(A)) \\
@A{A_i}AA & & @VV{A_i}V \\
H_+(i(A)) @<<< H_-(A) @<<< H_-(i(A))
\end{CD}
\end{equation}
Here the boundary maps are induced by the action of the linear operator $A_i$
on the Koszul complex of the commuting tuple $i(A)$. The six term exact
sequences obtained in this fashion can be thought of as analogs of the long
exact sequences \eqref{eq:akosz} and \eqref{eq:bkosz}.

Let us suppose that the commuting tuple $i(A)$ is Fredholm. This means that
the Koszul homology group $H(i(A))$ is a finite dimensional vector space. We
can then apply a determinant functor to our six term exact sequence in
homology \eqref{eq:sixiho}. After some canonical identifications this gives
rise to an isomorphism
\[
T_i(A) : \T{det}\big( H_+(A) \big) \to \T{det}\big( H_-(A) \big)
\]
between the determinants of the even and odd Koszul homology groups of the
commuting tuple $A$.

We could carry out the same construction for some fixed $j \in
\{1,\ldots,n\}$. Thus, if we assume that the commuting tuple $j(A)$ is
Fredholm as well, we get another isomorphism
\[
T_j(A) : \T{det}\big( H_+(A) \big) \to \T{det}\big( H_-(A) \big)
\]
between the same one dimensional vector spaces. The quotient of these two
isomorphism can therefore be identified with an invertible number
\[
\tau_{i,j}(A) = (-1)^{\mu_i(A) + \mu_j(A)} T_j(A)^{-1} \ci T_i(A) \in \ff^*
\]
This is the \emph{joint torsion transition number} in position $(i,j)$. The
exponents $\mu_i(A) \T{ and } \mu_j(A) \in \nn \cup \{0\}$ are given by
appropriate dimensions of Koszul homology groups. It can be proved that our
joint torsion recovers the Carey-Pincus joint torsion when the commuting tuple
$A$ consists of a pair of Fredholm operators. The joint torsion transition
numbers are the principal subject of the present paper. Let us state our main
results. The first one justifies the use of the word "transition" in our
definition. Indeed, the joint torsion transition numbers satisfy the same
relations as the transition functions of a line bundle.

\begin{prop}
Suppose that $k \in \{1,\ldots,n\}$ is an extra number such that the commuting
tuple $k(A) = (A_1,\ldots,\wih{A_k},\ldots,A_n)$ is Fredholm. The joint
torsion transition numbers then satisfy the cocycle property
\[
\arr{ccc}{
\tau_{i,j}(A) = \tau_{j,i}(A)^{-1} & & \tau_{i,j}(A) \cd \tau_{j,k}(A) = \tau_{i,k}(A)
}
\]
\end{prop}

The second one is a triviality property. It says that the joint torsion
transition number in position $(i,j)$ is equal to the identity when the
commuting tuples $i(A)$ and $j(A)$ are Fredholm for a trivial reason. For
example, in the case of a commuting pair of operators it means that the joint
torsion is trivial when the vector space is finite dimensional. As another
example, the theorem implies that the joint torsion transition number in
position $(i,j)$ is equal to one, when $A_k \in \C L(E)$ is a Fredholm
operator for some $k \neq i,j$.

\begin{prop}\label{triv}
The joint torsion transition number in position $(i,j)$ is trivial when the
Koszul homology group $H\big((ij)(A)\big)$ is finite dimensional. Here we let
\[
(ij)(A) := (A_1,\ldots,\wih{A_i},\ldots, \wih{A_j},\ldots,A_n)
\]
denote the commuting tuple obtained from $A$ by removing both of the operators
$A_i$ and $A_j$.
\end{prop}

The final result is a multiplicativity property. It can be understood as an
analog of the multiplicativity property for the Carey-Pincus joint torsion
stated in \eqref{eq:symbjoint}. We let $B = (A_1,\ldots,B_m,\ldots,A_n)$ be
another commuting tuple which only differs from $A = (A_1,\ldots, A_n)$ in the
$m^{\T{th}}$ coordinate. By the product of $A$ and $B$ we will then understand
the commuting tuple
\[
A \cd B = (A_1,\ldots,A_m \cd B_m,\ldots,A_n)
\]
Notice that we do not assume that the operators $A_m$ and $B_m$ commute.

\begin{prop}\label{mul}
Suppose that two of the three joint torsion transition numbers $\tau_{i,j}(A),
\tau_{i,j}(B)$ and $\tau_{i,j}(A \cd B)$ make sense. Then the third one is
also well-defined and related to the two others by the multiplicativity
relation
\[
\tau_{i,j}(A) \cd \tau_{i,j}(B) = \tau_{i,j}(A \cd B)
\]
\end{prop}

It might be worthwhile to discuss some aspects of the proofs of Theorem
\ref{triv} and Theorem \ref{mul}. The main tool is a comparison result for
determinants of certain triangles of chain complexes. These triangles appear
as the rows and columns of a larger diagram of chain complexes. This
comparison theorem lies at the technical core of the paper and we spend some
time giving a detailed proof. It should be noted that there most certainly
exists a link between our results and the general construction of determinant
functors of triangulated categories. See \cite{breuning,murotonkwit}. Making
this link explicit would furnish the abstract theory of determinants with a
concrete operator theoretic application. Furthermore, it would help
conceptualizing the direct approach which we apply in this article.

There are at least two other important issues which we do not treat in this
paper, but which we hope to address in the future. 

The first one is the relation between the joint torsion transition numbers and
algebraic $K$-theory. For example, suppose that the Koszul homology group
$H(A) = \{0\}$ is trivial. In this case, we give a formula for the joint
torsion transition numbers in terms of quotients of determinants. See Theorem
\ref{lefschetz}. The appearance of this multiplicative Lefschetz number
suggests that the joint torsion transition numbers could be obtained as values
of the determinant invariant on the second algebraic $K$-group. This is also
supported by Theorem \ref{jointstein} which states that this is the case in
the low-dimensional situation. The $K$-theoretic interpretation of the joint
torsion is currently under investigation.

The second issue is the extension of the Carey-Pincus multiplicative index
theorem to the case of Toeplitz operators over the polydisc. See Theorem
\ref{tametame}. This task is harder than finding a $K$-theoretic
interpretation of the joint torsion. Indeed, it is not completely clear what
the correct replacement of the tame symbol should be. However, let $f \in \C
A(U^n)$ be an invertible element of the polydisc algebra. We can then compute
the joint torsion transition number in position $(1,j)$ of the commuting tuple
\[
T_\al = (T_f,T_{z_1- \al_1},\ldots,T_{z_n - \al_n})
\]
of Toeplitz operators. Here $\al_1,\ldots,\al_n \in \B D^\ci:= U$ are complex
numbers in the interior of the disc and $z_1,\ldots,z_n : \B T^n \to \cc$ are
the coordinate functions on the $n$-torus. In this case, the joint torsion is
simply given by the evaluation of $f \in \C A(U^n)$ at the point $\al =
(\al_1,\ldots,\al_n) \in U^n$,
\[
\tau_{1,j}(T_\al) = f(\al) \in \cc^*
\]
This result is obtained in Theorem \ref{jointcauchy}.

{\bf Acknowledgements:} First of all, I should thank the Australian Research
Council for supporting me while I was at the Australian National University
(ANU). I did most of the research for the present article during my stay there
in the southern automn of $2010$. Secondly, I would like to thank the people
at the ANU. In particular, I should mention Mike Eastwood who remarked that my
first definition of joint torsion was probably incorrect because of the lack
of symmetry. I am of course grateful for this valuable remark. I should also
mention Greg Stevenson who gave me a short introduction to triangulated
categories and with whom I in general had many good discussions. I would also
like to thank Alan Carey, Amnon Neeman and Adam Rennie for showing interest in
my work and for their helpful comments in this regard. Next, I would like to
thank Jerome Kaminker for drawing my attention to the papers of Richard Carey
and Joel Pincus when we met at the U.C. Davis in the northern spring of
$2008$. And finally, as always, I would like to thank Ryszard Nest for his
continuous support and willingness to share some of his inspiring thoughts on
various subjects.

\section{Torsion isomorphisms}
In this section we will construct torsion isomorphisms in different
contexts. The first subsection is concerned with the most basic situation. We
look at an odd endomorphism of a $\zz_2$-graded vector space which satisfies
an exactness condition. We then associate a determinant to such an
endomorphism. The second subsection is concerned with determinants of odd
exact triangles of finite dimensional vector spaces. We use the construction
of the first subsection to give a definition of a determinant in this
situation. The last subsection is concerned with determinants of odd triangles
of chain complexes which satisfy a homotopy exactness condition. Furthermore,
these chain complexes are supposed to have finite dimensional homology. We use
the construction of the second subsection to give a definition of a
determinant in this situation.

It should be noted that the material of this section is rather basic but also
absolutely essential for the rest of the paper. We should also stress the
relation to the construction of determinant functors in various settings one
more time. Thus, we refer to the papers
\cite{breuning,deligneII,knudsmumf,knuds,murotonkwit}. However, our exposition
is very simple minded and does not rely on any of these more elaborated
results.

\subsection{The torsion isomorphism of an odd exact endomorphism}
Let $\ff$ be a field and let $V$ be an $n$-dimensional vector space over
$\ff$. We let $\T{det}(V)$ denote the top part of the exterior algebra over
$V$,
\[
\arr{ccc}{
\T{det}(V):= \La_n(V) & \q & n = \T{dim}(V)
}
\]
This one-dimensional vector space will be referred to as the \emph{determinant
  of $V$}. Furthermore, we let $V^* := \T{Hom}(V, \ff)$ denote the dual of the
vector space $V$.

Now, let $E_+$ and $E_-$ be two vector spaces over $\ff$ of the same finite
dimension. Let $\si : E_+ \to E_-$ be an isomorphism. We will then let
\[
\T{det}(\si) : \T{det}(E_+) \to \T{det}(E_-)
\]
denote the isomorphism obtained by functoriality of the exterior power. This
isomorphism will be referred to as the \emph{determinant of $\si$}. In the
case where $E_+ = E_-$ we can identify the determinant of $\si$ with an
invertible element in the field $\ff$.

Let us look at the $\zz_2$-graded finite dimensional vector space
\[
E := E_+ \op E_-
\]
given by the direct sum of $E_+$ and $E_-$.

Let $\al : E \to E$ be an odd endomorphism, thus $\al \in \T{End}_-(E)$ is
given by a matrix
\[
\al = \mat{cc}{
0 & \al_- \\
\al_+ & 0
} : E_+ \op E_- \to E_+ \op E_-
\]
relative to the decomposition $E = E_+ \op E_-$.

\begin{definition}
We will say that $\al \in \E{End}_-(E)$ is \emph{exact} when the kernel of
$\al_+$ agrees with the image of $\al_-$ and vice versa, thus
\[
\arr{ccc}{
\E{Ker}(\al_+) = \E{Im}(\al_-) & \q & \E{Ker}(\al_-) = \E{Im}(\al_+)
}
\]
\end{definition} 

Suppose that $\al \in \T{End}_-(E)$ is exact and let us choose a
pseudo-inverse $\al_-^\da : E_+ \to E_-$ of $\al_- : E_- \to E_+$. Recall that
this corresponds to the choice of algebraic decompositions
\[
\arr{ccc}{
E_+ \cong \T{Im}(\al_-) \op Q_- & \T{and} & E_- \cong \T{Ker}(\al_-) \op C_-
}
\]
It then follows from the exactness of $\al \in \T{End}_-(E)$ that the linear
map
\[
\al_+ + \al_-^\da : E_+ \to E_-
\]
is an isomorphism.

\begin{definition}
By the \emph{torsion isomorphism} of $\al \in \E{End}_-(E)$ we will understand
the determinant of the isomorphism $\al_+ + \al_-^\da : E_+ \to E_-$. The
torsion isomorphism of $\al$ will be denoted by $T(\al)$, thus
\[
T(\al) := \E{det}(\al_+ + \al_-^\da) \in 
\E{Hom}\big(\E{det}(E_+), \E{det}(E_-)\big) - \{0\}
\]
\end{definition}

Notice that the torsion isomorphism is not a non-zero element in the field
$\ff$ but only a non-zero vector in a one-dimensional vector space over
$\ff$.

\begin{lemma}\label{tortwo}
The torsion isomorphism of $\al$ is independent of the choice of
pseudo-inverse of $\al_-$.
\end{lemma}
\begin{proof}
Let $\al^*_- : E_+ \to E_-$ be a different choice of a pseudo inverse for
$\al_- : E_- \to E_+$. Let us use the notation
\[
\arr{ccc}{
e = \al_-^\da \al_- \in \C L(E_-) & \q & p = \al_- \al_-^\da  \in \C L(E_+) \\
f = \al_-^* \al_- \in \C L(E_-) & \q & q = \al_- \al_-^* \in \C L(E_+)
}
\]
for the idempotents associated with the different choices of pseudo-inverses.

We then have the relation
\[
\al_+ + \al_-^* 
= \al_+ + f \al_-^\da q
= (1 - e + f)(\al_+ + \al_-^\da)(1 - p + q)
\]
between the different isomorphisms. In particular, we get the relation
\[
\T{det}(\al_+ + \al_-^*)
= \T{det}(1 - e + f) \cd \T{det}(\al_+ + \al_-^\da) \cd \T{det}(1 - p + q)
\]
between the different determinants. The result of the lemma is now a
consequence of the identities
\[
\T{det}(1 - e + f) = 1 = \T{det}(1 - p + q)
\]
which can be easily verified.
\end{proof}



\subsection{The torsion isomorpism of an odd exact triangle}\label{torextri}
Let us consider an odd triangle of vector spaces
\[
\begin{CD}
V \, \, : \, \, V^1 @>v^1>> V^2 @>v^2>> V^3 @>v^3>> V^1
\end{CD}
\]
This means that the vector spaces 
\[
V^i = V^i_+ \op V^i_-
\]
are $\zz_2$-graded for all $i \in \{1,2,3\}$ and that the linear maps
\[
\arr{ccccc}{
v^1 : V^1 \to V^2 & & v^2 : V^2 \to V^3 & \T{and} & v^3 : V^3 \to V^1
}
\]
are odd.

\begin{definition}
We will say that the odd triangle $V$ is \emph{exact}, when the sequence of
six terms
\[
\begin{CD}
V^1_+ @>{v^1_+}>> V^2_- @>{v^2_-}>> V^3_+ \\
@A{v^3_-}AA & & @V{v^3_+}VV \\
V^3_- @<<{v^2_+}< V^2_+ @<<{v^1_-}< V^1_- 
\end{CD}
\]
is exact.
\end{definition}

We can form a $\zz_2$-graded vector space out of the odd triangle $V$. We will
also use the letter $V := V_+ \op V_-$ for this vector space. It is given by
the positive and negative components
\begin{equation}\label{eq:decomptri}
\arr{ccc}{
V_+ := V^1_+ \op V^2_+ \op V^3_+ 
& \T{and} & V_- := V^1_- \op V^2_- \op V^3_-
}
\end{equation}
We then get an odd endomorphism
\begin{equation}\label{eq:oddhomtri}
\arr{ccc}{
v : V \to V & & v :=
\mat{cc}{
0 & v_- \\
v_+ & 0
}
}
\end{equation}
defined by the matrices
\[
\arr{ccc}{
v_+ := \mat{ccc}{
0 & 0 & v^3_+ \\
v^1_+ & 0 & 0 \\
0 & v^2_+ & 0
}
& \T{and} &
v_- := \mat{ccc}{
0 & 0 & v^3_- \\
v^1_- & 0 & 0 \\
0 & v^2_- & 0
}
}
\]
Here we use the decompositions of $V_+$ and $V_-$ given by
\eqref{eq:decomptri}.

The following lemma is then an immediate consequence of the definitions.

\begin{lemma}\label{triexact}
The odd homomorphism $v \in \E{End}_-(V)$ is exact if and only if the odd
triangle $V$ is exact.
\end{lemma}

Suppose that our odd triangle $V$ is exact and that the vector space $V^i$ is
finite dimensional for all $i \in \{1,2,3\}$. By combining Lemma
\ref{triexact} and Lemma \ref{tortwo} we can give a definition of the torsion
isomorphism associated to $V$.

\begin{definition}
By the \emph{torsion isomorphism} of the odd exact triangle $V$ we will
understand the torsion isomorphism of the odd exact endomorphism $v \in
\E{End}_-(V)$. The torsion isomorphism of $V$ will be denoted by $T(V)$, thus
by definition
\[
T(V) := T(v) \in \E{Hom}\big(\E{det}(V_+), \E{det}(V_-) \big) - \{0\}
\]
\end{definition}



\subsection{The torsion isomorphism of an odd homotopy exact triangle}
Let us consider a $\zz_2$-graded chain complex, $X$. Thus, $X = X_+ \op X_-$
is a $\zz_2$-graded vector space equipped with an odd homomorphism
\[
\arr{ccc}{
d : X \to X & & d := \mat{cc}{
0 & d_- \\
d_+ & 0
}
}
\]
with square equal to zero
\[
d^2 = \mat{cc}{d_- d_+ & 0 \\
0 & d_- d_+
} = 0
\]
The homology of $X$ is given by the $\zz_2$-graded vector space
\[
H(X) = \T{Ker}(d) / \T{Im}(d) = H_+(X) \op H_-(X)
\]
where the components are defined by
\[
\arr{ccc}{
H_+(X) := \T{Ker}(d_+)/\T{Im}(d_-) & & H_-(X) := \T{Ker}(d_-)/ \T{Im}(d_+)
}
\]
We will refer to the odd homomorphism $d : X \to X$ as the differential on
$X$.

\begin{definition}\label{trichain}
By an \emph{odd triangle of chain complexes} we will understand an odd
triangle of vector spaces
\[
\begin{CD}
X \, \, : \, \, X^1 @>{v^1}>> X^2 @>{v^2}>> X^3 @>{v^3}>> X^1
\end{CD}
\]
such that
\begin{enumerate}
\item Each $\zz_2$-graded vector space $X^i$ is a $\zz_2$-graded chain
  complex.
\item The odd linear maps
\[
\arr{ccccc}{
v^1 : X^1 \to X^2 & & v^2 : X^2 \to X^3 & & v^3 : X^3 \to X^1
}
\]
are odd chain maps. This means that they are odd homomorphisms of
$\zz_2$-graded vector spaces which \emph{anti-commute} with the
differentials.
\end{enumerate}
\end{definition}

Let us fix an odd triangle of chain complexes
\[
\begin{CD}
X \, \, : \, \, X^1 @>{v^1}>> X^2 @>{v^2}>> X^3 @>{v^3}>> X_1
\end{CD}
\]
By passing to homology we get an induced odd triangle of vector spaces
\[
\begin{CD}
H(X) \, \, : \, \, H(X^1) @>{v^1}>> H(X^2) @>{v^2}>> H(X^3) @>{v^3}>>
H(X^1)
\end{CD}
\]
We will now give a sufficient condition for the exactness of this odd
triangle at the level of homology.

\begin{definition}\label{homoex}
We will say that the odd triangle of chain complexes $X$ is \emph{homotopy
  exact} when there exist odd linear maps
\[
\arr{ccccc}{
t^1 : X^2 \to X^1 & & t^2 : X^3 \to X^2 & & 
t^3 : X^1 \to X^3
}
\]
which satisfy the conditions
\begin{enumerate}
\item The maps define \emph{chain homotopies} between $0$ and the squares of
  the odd chain maps associated with $X$. Thus, we have the identity
\[
d^{i-1} t^{i-1} + t^{i-1} d^i = v^{i+1} v^i : X^i \to X^{i-1}
\]
for all $i \in \{1,2,3\}$. Here the notation $d^i : X^i \to X^i$ refers to the
differential on $X^i$. Notice that we calculate with the indices modulo
three.
\item The maps define a \emph{homotopy decomposition} of the chain complexes
  in the sense that the chain maps
\[
v^{i-1} t^{i-1} +  t^i v^i : X^i \to X^i
\]
induce the identity at the level of homology for all $i \in \{1,2,3\}$. Remark
that the chain map property follows from the above chain homotopy condition.
\end{enumerate}
We will refer to the opposite odd triangle of vector spaces
\[
\begin{CD}
X^\da \, \, : \, \, X^1 @<<{t^1}< X^2 @<<{t^2}< X^3 @<<{t^3}< X^1
\end{CD}
\]
as a \emph{homotopy} for $X$.
\end{definition}

As promised above we then have the following lemma.

\begin{lemma}\label{exhomol}
If the odd triangle of chain complexes $X$ is homotopy exact, then the odd
triangle of vector spaces $H(X)$ is exact.
\end{lemma}
\begin{proof}
We need to prove that the induced sequence of six terms
\[
\begin{CD}
H_+(X^1) @>{v^1}>> H_-(X^2) @>{v^2}>> H_+(X^3) \\
@AA{v^3}A & & @V{v^3}VV \\
H_-(X^3) @<<{v^2}< H_+(X^2) @<<{v^1}< H_-(X^1)
\end{CD}
\]
is exact.

Let $i \in \{1,2,3\}$. First of all the identity
\[
v^{i+1} v^i = 0 \, \, : \, \, H(X^i) \to H(X^{i-1})
\]
follows by noting that $t^{i-1} : X^i \to X^{i-1}$ defines a homotopy
between $v^{i+1} v^i : X^i \to X^{i-1}$ and zero.

Now, assume that $v^i[x] = 0$ for some $[x] \in H(X^i)$. Let $x \in X^i$ be a
representative of $[x] \in H(X^i)$ and let $y \in X^{i+1}$ be an element with
$d^{i+1}(y) = v^i(x)$. Notice that $d^i(x) = 0$ by assumption. We define an
element $z \in X^{i-1}$ by the formula
\[
z := v^{i+1}(y) + t^{i-1}(x)
\]
We then have that
\[
\begin{split}
d^{i-1}(z) 
& = (d^{i-1} v^{i+1})(y) + (d^{i-1} t^{i-1})(x) \\
& = - (v^{i+1}d^{i+1}) (y) + (v^{i+1} v^i)(x) \\
& = 0
\end{split}
\]
Thus, $z \in X^{i-1}$ defines a class $[z] \in H(X^{i-1})$. We can then
calculate as follows
\[
\begin{split}
v^{i-1}[z]
& = \big[ (v^{i-1} v^{i+1})(y) + (v^{i-1} t^{i-1})(x)\big] \\
& = \big[ (t^i d^{i+1})(y) + (v^{i-1} t^{i-1})(x) \big] \\
& = \big[ (t^i v^i)(x) + (v^{i-1}t^{i-1})(x)\big] \\
& = [x]
\end{split}
\]
in the homology group $H(X^i)$. This proves the lemma.
\end{proof}

Suppose that $X$ is an odd homotopy exact triangle of chain
complexes. Furthermore, suppose that the induced odd exact triangle $H(X)$
consists of finite dimensional vector spaces.

\begin{definition}\label{torhomex}
By the \emph{torsion isomorphism} of the odd homotopy exact triangle $X$ we
will understand the torsion isomorphism of the odd exact triangle induced by
$X$ at the level of homology. Thus, by definition
\[
T(X) := T\big(H(X)\big) \in 
\E{Hom}\big(\E{det}(H_+(X)), \E{det}(H_-(X))\big) - \{0\}
\]
Here we recall that $H(X)$ also denotes the $\zz_2$-graded vector space given
by the components
\[
\begin{split}
H_+(X) & := H_+(X^1) \op H_+(X^2) \op H_+(X^3)
\q \T{and} \\
H_-(X) & := H_-(X^1) \op H_-(X^2) \op H_-(X^3)
\end{split}
\]
See Section \ref{torextri}.
\end{definition}





\section{Joint torsion of commuting operators}\label{jtorsion}
The main focus of this section lies on the construction of the joint torsion
transition numbers of a commuting tuple of linear operators satisfying a
Fredholm condition. In the first subsection we recall the definition of the
Koszul complex of a commuting tuple of linear operators. In the second section
we will review some results from the Fredholm theory in several variables. We
will do this in a purely algebraic context. In the third section we construct
the joint torsion transition numbers and prove some of their basic
properties. This section lies at the heart of the present article. In the
fourth and fifth section we compute the joint torsion transition numbers in
some concrete examples. The first example is concerned with the situation
where the Koszul homology of our commuting tuple is trivial. In this setup we
obtain a multiplicative Lefschetz number as the value of our joint torsion
invariant. The framework for the second example is the theory of Toeplitz
operators over the polydisc. The calculation which we present here provides a
link between the joint torsion transition numbers and the Cauchy integral
formula for holomorphic functions on the polydisc. In particular, we get an
alternative proof of the fact, that a continuous function on the torus has at
most one holomorphic extension to the polydisc.

Let us start by indicating a sign convention which we will use throughout. Let
$V,W,V_1,\ldots,V_m$ be finite dimensional vector spaces and let $\si \in
\Si_m$ be a permutation. We can then form the finite dimensional vector spaces
\[
\arr{ccc}{
F_V := V \op \big(V_1 \oplop V_m \big)
& \T{and} & F_W := W \op \big(V_{\si(1)} \oplop V_{\si(m)}\big)
}
\]
We will identify the one-dimensional vector spaces
\[
\arr{ccc}{
\T{Hom}(\T{det}(F_V),\T{det}(F_W)) & \T{and} &
\T{Hom}(\T{det}(V),\T{det}(W))
}
\]
using a string of isomorphisms. To be precise, we have that
\begin{equation}\label{eq:canonisoiso}
\begin{split}
\T{Hom}(\T{det}(F_V), \T{det}(F_W))
& \cong \T{det}(F_V)^* \ot \T{det}(F_W) \\
& \cong \T{det}(V)^* \ot \T{det}(V_1)^* \olo \T{det}(V_m)^* \\
& \q \ot \T{det}(W) \ot \T{det}(V_{\si(1)}) \olo \T{det}(V_{\si(m)})
\\
& \cong \T{det}(V)^* \ot \T{det}(W) \\
& \cong \T{Hom}(\T{det}(V), \T{det}(W))
\end{split}
\end{equation}
Here the second last isomorphism is given by
\[
\ga \ot (\ga_1 \olo \ga_m) \ot x \ot (x_1 \olo x_m)
\mapsto (\ga \ot x) \cd \ga_{\si(1)}(x_1) \clc \ga_{\si(m)}(x_m)
\]
on simple tensors. It is important to notice that the sign of the permutation
does not show up in this formula.

Now, let us fix a field $\ff$ and a natural number $n$. We let $A =
(A_1,\ldots,A_n) \in \C L(E)^n$ be a commuting tuple of linear operators on
some vector space $E$ over $\ff$. Thus, we have the relation
\[ 
\arr{ccc}{
A_i A_j - A_j A_i = 0 & \T{for all} & i,j \in \{1,\ldots,n\} 
}
\]

\subsection{The Koszul complex}
We let $\La(\ff^n)$ denote the exterior algebra over $\ff^n$. The exterior
product will be denoted by
\[
\we : \La(\ff^n) \ti \La(\ff^n) \to \La(\ff^n)
\]
We will think of the exterior algebra as a $\zz_2$-graded vector space with
grading given by even and odd exterior powers. Thus,
\[
\arr{ccc}{
\La(\ff^n) = \La_+(\ff^n) \op \La_-(\ff^n)
}
\]
where $\La_+(\ff^n) := \bop_{m = 2k} \La_m(\ff^n)$ and $\La_-(\ff^n) :=
\bop_{m= 2k+1} \La_m(\ff^n)$. 

Let $j \in \{1,\ldots,n\}$ and let us look at the inclusion
\[
\arr{ccc}{
\io_j : \ff^{n-1} \to \ff^n & & 
\io_j(\la_1,\ldots,\la_{n-1}) = 
(\la_1,\ldots,\la_{j-1},0,\la_j,\ldots,\la_{n-1})
}
\]
This linear map induces an even algebra homomorphism
\[
\io_j : \La(\ff^{n-1}) \to \La(\ff^n)
\]
at the level of exterior algebras.

We could also look at the projection
\[
\arr{ccc}{
\io_j^* : \ff^n \to \ff^{n-1}
& & \io_j^*(\la_1,\ldots,\la_n) = 
(\la_1,\ldots,\la_{j-1},\la_{j+1},\ldots,\la_n)
}
\]
This linear map induces an even algebra homomorphism
\[
\io_j^* : \La(\ff^n) \to \La(\ff^{n-1})
\]
at the level of exterior algebras.

Next, we let
\[
\arr{ccc}{
\ep_j := e_j \we \cd \, \, : \, \, \La(\ff^n) \to \La(\ff^n) 
}
\]
denote the odd linear operator given by exterior multiplication from the left
by the $j^{\T{th}}$ standard basis vector $e_j \in \ff^n$. Finally, we let
\[
\arr{ccc}{
\ep_j^* : \La(\ff^n) \to \La(\ff^n)
}
\]
denote the odd linear operator given by interior multiplication by the
$j^{\T{th}}$ standard dual basis vector $e_j^* \in (\ff^n)^*$. Thus, to be
explicit we have the formula
\[
\ep_j^* : e_{i_1} \wlw e_{i_k} \mapsto \fork{ccc}{
0 
& \T{for} & j \notin \{i_1,\ldots,i_k\} \\
(-1)^{m-1} e_{i_1} \wlw \wih{e_{i_m}} \wlw e_{i_k}
& \T{for} & j = i_m
}
\]
on the standard basis vectors of the exterior algebra $\La(\ff^n)$.

Now, let us look at the commuting tuple of linear operators $A =
(A_1,\ldots,A_n) \in \C L(E)^n$. By the \emph{Koszul complex of $A$} we will
then understand the $\zz_2$-graded chain complex given by the following data:
\begin{enumerate}
\item The $\zz_2$-graded vector space
\[
K(A) := E \ot_\ff \La(\ff^n) 
\]
Here the grading is given by the even and odd components
\[
\arr{ccc}{
K_+(A) := E \ot_\ff \La_+(\ff^n) & \T{and} & 
K_-(A) := E \ot_\ff \La_-(\ff^n)
}
\]
\item The differential is given by
\[
d^A := \sum_{j=1}^n A_j \ot \ep_j^* : K(A) \to K(A)
\]
\end{enumerate}
Notice that $(d^A)^2 = 0$ since the linear operators $A_1,\ldots,A_n \in \C
L(E)$ commute whereas the interior multiplication operators anti-commute. We
will use the notation $K(A)$ for the Koszul complex and the notation $H(A) =
H_+(A) \op H_-(A)$ for the $\zz_2$-graded homology of $K(A)$.

The Koszul complex will play a central role in the present paper.



\subsection{Fredholm theory in several variables}\label{fredthe}
The Koszul complex can be used to generalize the Fredholm theory of linear
operators on vector spaces to a multivariable setting.

\begin{definition}
We will say that the commuting tuple $A = (A_1,\ldots,A_n)$ is \emph{Fredholm}
when the homology of the associated Koszul complex is finite dimensional. In
this case we define the \emph{Fredholm index of $A$} as minus the Euler
characteristic of the Koszul complex. Thus,
\[
\E{Ind}(A) := - \chi\big( K(A)\big) = \E{dim}\big(H_-(A)\big) -
\E{dim}\big(H_+(A)\big)
\]
Clearly, the index is an integer. 
\end{definition}

The Fredholm index has a couple of good algebraic properties which we will now
state. First of all the Fredholm index is symmetric. Let $\si \in \Si_n$ be a
permutation and let $\si(A)$ denote the commuting tuple
\[
\si(A) := (A_{\si(1)},\ldots, A_{\si(n)})
\]

\begin{prop}\label{indsymm}
The commuting tuple $\si(A)$ is Fredholm if and only if the commuting tuple
$A$ is Fredholm. In this case the two indices coincide,
\[
\E{Ind}(A) = \E{Ind}\big( \si(A) \big)
\]
Notice that the sign of the permutation \emph{does not} show up in this
formula.
\end{prop}
\begin{proof}
This follows immediately by noting that the two Koszul complexes are
isomorphic by the even chain map
\[
1 \ot \si^{-1} : K(A) \to K\big(\si(A)\big)
\]
Here $\si^{-1} : \La(\ff^n) \to \La(\ff^n)$ is the even algebra isomorphism
defined on generators by
\[
\si^{-1} : e_j \mapsto e_{\si^{-1}(j)}
\]
\end{proof}

Secondly, the Fredholm index satisfies a triviality property.

\begin{prop}\label{indtriv}
Suppose that $A$ is Fredholm and let $A_{n+1} \in \C L(E)$ be an extra linear
operator which commutes with all the operators $A_1,\ldots,A_n \in \C
L(E)$. Then the commuting tuple $A \cup A_{n+1} := (A_1,\ldots,A_{n+1})$ is
Fredholm and the index is trivial,
\[
\E{Ind}(A \cup A_{n+1}) = 0
\]
\end{prop}
\begin{proof}
This will turn out to be a consequence of Theorem \ref{minusex} and general
properties of the Euler characteristic.
\end{proof}

Finally, the Fredholm index is additive in each variable. Let $B =
(B_1,A_2,\ldots,A_n)$ be a commuting tuple of linear operators which only
differ from $A$ in the first coordinate. We then define the product $A \cd B$
as the commuting tuple
\[
A \cd B := (A_1\cd B_1, A_2,\ldots, A_n)
\]

\begin{prop}\label{multlin}
If two of the three commuting tuples $A, B$ and $A \cd B$ are Fredholm then
the third one is also Fredholm. In this case the indices satisfy the
additivity relation
\[
\E{Ind}(A \cd B) = \E{Ind}(A) + \E{Ind}(B)
\]
\end{prop}
\begin{proof}
This will turn out to be a consequence of Lemma \ref{multtri} and general
properties of the Euler characteristic.
\end{proof}

Remark that a similar additivity result holds for the other entries. This
follows immediately from Theorem \ref{multlin} by the symmetry property of the
Fredholm index, see Theorem \ref{indsymm}.
 
The theorems stated above are only modifications of known results. Let us give
some appropriate references.

A result which is similar to Theorem \ref{indsymm} has been proved by R. Curto
\cite[Proposition $9.6$]{curto}. Next, a result which is similar to Theorem
\ref{indtriv} has also been proved by R. Curto \cite[\S $8
(i)$]{curto}. However, the proof of the triviality result relies on the
homotopy invariance of the Fredholm index in the setting of Hilbert
spaces. The proof can therefore not be directly translated to the algebraic
setup we are considering in the present paper. Finally, a result which is
similar to Theorem \ref{multlin} has been proved by X. Fang \cite[Proposition
$1$]{fang}. However, it requires the additional assumption that $A_1$ and
$B_1$ commute. Our Theorem \ref{multlin} can thus be regarded as a slight
improvement of \cite[Proposition $1$]{fang}.

\begin{remark}
When $E$ is a Hilbert space and the commuting operators $A_1,\ldots,A_n \in \C
L(E)$ are bounded, there is an interpretation of the Fredholmness condition in
terms of the joint essential spectrum of the commuting tuple $A =
(A_1,\ldots,A_n)$, see \cite[\S $6$ Theorem $2$]{curto}. It follows that the
Fredholmness condition is well behaved under the analytic functional calculus
as developed by J. Taylor in \cite{taylorI, taylorII}. Furthermore, in this
setup, the Fredholm index is invariant under homotopies and compact
perturbations, see \cite[\S $7$ Theorem $3$]{curto}. We will apply some of
these more analytic results in the examples even though the main focus will be
on the purely algebraic aspects of the theory.
\end{remark}

The joint torsion, which we will soon define, can be understood as a
multiplicative analog of the Fredholm tuple index.



\subsection{The joint torsion transition numbers}\label{jointtrans}
Let $j \in \{1,\ldots,n\}$. We will use the notation
\[
j(A) := (A_1,\ldots,\wih{A_j}, \ldots,A_n)
\]
for the commuting tuple obtained from the commuting tuple $A =
(A_1,\ldots,A_n)$ by removing the entry in position $j$.

The linear operator $A_j \in \C L(E)$ then acts on the Koszul complex of
$j(A)$ by means of the chain map
\[
A_j := A_j \ot 1 : K\big( j(A)\big) \to K\big( j(A) \big)
\]
Furthermore, the Koszul complex of $j(A)$ can be regarded as a sub-complex of
the Koszul complex of $A$ by means of the chain map
\begin{equation}\label{eq:abbrevI}
\io_j := 1 \ot \io_j : K\big( j(A)\big) \to K(A)
\end{equation}
Finally, the shifted Koszul complex of $j(A)$ can be regarded as a quotient of
the Koszul complex of $A$ by means of the chain map
\begin{equation}\label{eq:abbrevII}
\io_j^* \ep_j^* := 1 \ot \io_j^* \ep_j^* : K(A) \to K\big( j(A) \big)[1]
\end{equation}
Here the notation $X[1]$ refers to the $\zz_2$-graded chain complex obtained
from some $\zz_2$-graded chain complex $X$ by reversing the grading and
changing the sign of the differential.

In particular, we have an odd triangle of chain complexes in the sense of
Definition \ref{trichain},
\[
\begin{CD}
X^A_j \, \, : \, \, 
K\big(j(A) \big) @>{A_j}>>
K\big( j(A) \big)[1] @>{\io_j}>>
K(A) @>{\io_j^* \ep_j^*}>>
K\big(j(A) \big)
\end{CD}
\]

We would like to prove that the triangle $X^A_j$ is homotopy exact in the
sense of Definition \ref{homoex}. To this end we define the opposite triangle
\[
\begin{CD}
(X^A_j)^\da \, \, : \, \,
K\big(j(A) \big) @<<{0}<
K\big( j(A) \big)[1] @<<{\io_j^*}<
K(A) @<<{\ep_j \io_j}<
K\big(j(A) \big)
\end{CD}
\]
Here the odd linear maps included are abbreviated in the same way as in
\eqref{eq:abbrevI} and \eqref{eq:abbrevII}.

\begin{prop}\label{minusex}
The triangle $X^A_j$ is homotopy exact with homotopy given by the opposite
triangle $(X^A_j)^\da$.
\end{prop}
\begin{proof}
Let us start by checking the identities
\[
\arr{ccccc}{
0 
= \io_j^* \ep_j^* \io_j & \q &
-d^{j(A)} \io_j^* + \io_j^* d^A 
= A_j \io_j^* \ep_j^* & \q &
d^A \ep_j \io_j + \ep_j \io_j d^{j(A)}
= \io_j A_j
}
\]
The first identity is immediate. The identity in the middle can be proved by
an application of the relations
\[
\ep_i^* \io_j^* = \fork{ccc}{
\io_j^* \ep_i^* & \T{for} & i < j \\
\io_j^* \ep_{i+1}^* & \T{for} & i \geq j
}
\]
between the projections and the interior multiplication operators. The last
identity can be proved by an application of the relations
\[
\ep_i^* \ep_j = \fork{ccc}{
- \ep_j \ep_i^* & \T{for} & i \neq j \\
1 - \ep_j \ep_j^* & \T{for} & i = j
}
\]
between the exterior and interior multiplication operators. Furthermore, it is
convenient to use that $\io_j : K\big( j(A)\big) \to K(A)$ is a chain map.

Next, we should prove the identities
\[
\arr{ccccc}{
\io_j^* \io_j = 1 & \q & 
\io_j \io_j^* + \ep_j \io_j \io_j^* \ep_j^* = 1 & \q &
\io_j^* \ep_j^* \ep_j \io_j = 1
}
\]
at the level of homology. However, it is not hard to see that they are
actually valid at the level of chain complexes.
\end{proof}

As a consequence of Theorem \ref{minusex} and Lemma \ref{exhomol} we get a six
term exact sequence of homology groups
\[
\begin{CD}
H_+\big(j(A)\big) @>{A_j}>> H_+\big(j(A)\big) @>{\io_j}>> H_+(A) \\
@A{\io_j^*\ep_j^*}AA & & @VV{\io_j^*\ep_j^*}V \\
H_-(A) @<<{\io_j}< H_-\big(j(A)\big) @<<{A_j}< H_-\big(j(A)\big)
\end{CD}
\]

Suppose now that the commuting tuple $j(A)$ is Fredholm. By definition, this
means that the $\zz_2$-graded homology group $H\big(j(A)\big)$ is finite
dimensional. It then follows from our six term exact sequence that the
homology group $H(A)$ is finite dimensional as well. Or in other words, the
commuting tuple $A$ is Fredholm. Furthermore, the index of $A$ is trivial,
$\T{Ind}(A) = 0$. Notice that this gives a proof of Theorem \ref{indtriv}.

However, we can can also use our six term exact sequence of homology groups to
construct a torsion isomorphism
\[
T(X^A_j) : \T{det}\big(H_+(X^A_j)\big) \to \T{det}\big(H_-(X^A_j)\big)
\]
See Definition \ref{torhomex}. Let us recall that the finite dimensional
vector spaces $H_+(X^A_j)$ and $H_-(X^A_j)$ are given by
\[
\begin{split}
H_+(X^A_j) & =
H_+\big(j(A) \big) \op H_-\big(j(A)\big) \op H_+\big( A \big) \q \T{and } \\
H_-(X^A_j) & =
H_-\big(j(A) \big) \op H_+\big(j(A)\big) \op H_-\big( A \big)
\end{split}
\]
See \eqref{eq:decomptri}. Now, the appearance of the homology group
$H\big(j(A) \big)$ as a direct summand in both of the above vector spaces
allows us to factor out this component from our torsion isomorphism. Indeed,
using an isomorphism which is similar to the isomorphism in
\eqref{eq:canonisoiso}, we get that
\begin{equation}\label{eq:isomdethom}
\T{Hom}\Big( 
\T{det}\big(H_+(X^A_j)\big), \T{det}\big(H_-(X^A_j)\big)\Big)
\cong \T{Hom}\Big(\deht{+}{A},\deht{-}{A}\Big)
\end{equation}

We will use the notation
\[
T_j(A) \in \T{Hom}\Big(\deht{+}{A},\deht{-}{A}\Big) - \{0\}
\]
for the isomorphism obtained from the torsion isomorphism
\[
T(X^A_j) \in \T{Hom}\Big(
\deht{+}{X^A_j}, \deht{-}{X^A_j} \Big) - \{0\}
\]
by applying the isomorphism of \eqref{eq:isomdethom}.

Now, let $i \in \{1,\ldots,n\}$ and suppose that the commuting tuple
\[
i(A) = (A_1,\ldots, \wih{A_i},\ldots,A_n)
\]
is Fredholm as well. Applying the above constructions we obtain another
isomorphism
\[
T_i(A) : \deht{+}{A} \to \deht{-}{A}
\]
We can thus form the composition
\[
T_j(A)^{-1} \ci T_i(A) \in \T{Aut}\big(\deht{+}{A}\big)
\]
which is an automorphism of the one-dimensional vector space
$\deht{+}{A}$. Or in other words, we can associate an invertible number to the
pair $(i,j)$ and the commuting tuple $A$.

\begin{definition}\label{defjoint}
By the \emph{joint torsion transition number in position $(i,j)$} of the
commuting tuple $A$ we will understand the non-zero number
\[
\tau_{i,j}(A) \in \ff^* 
\]
obtained from the automorphism
\[
(-1)^{\mu_j(A) + \mu_i(A)} T_j(A)^{-1} \ci T_i(A) 
\in \E{Aut}\big( \dehe{+}{A}\big) \cong \ff^*
\]
Here the exponent for the sign is given by products of dimensions of homology
groups
\begin{equation}\label{eq:signdef}
\arr{ccc}{
\mu_k(A) := \E{dim}\big(H_+(k(A)) \big) \cd \E{dim}\big(H_-(k(A)) \big) \in
\nn \cup \{0\} & & k = i,j
}
\end{equation}
\end{definition}

It should be remarked that in the case $n=2$, that is, when the commuting
tuple $A$ consists of two commuting Fredholm operators we recover the
Carey-Pincus joint torsion. Thus, we have the identity
\begin{equation}\label{eq:carpinckaad}
\tau_{1,2}(A_1,A_2) = \tau(A_1,A_2)
\end{equation}
Here $\tau(A_1,A_2) \in \ff^*$ denotes the joint torsion defined by R. Carey
and J. Pincus in \cite{carpincI}. This relation can be verified by carefully
keeping track of the signs in the two different definitions.

The use of the word "transition" in Definition \ref{defjoint} is justified by
the next lemma.

\begin{lemma}\label{trans}
Suppose that the commuting tuples $i(A)$, $j(A)$ and $k(A)$ are Fredholm for
three numbers $i,j,k \in \{1,\ldots,n\}$. Then the joint torsion transition
numbers satisfy the transition identities
\[
\arr{ccc}{
\tau_{i,j}^{-1}(A) = \tau_{j,i}(A) 
& \T{and} & \tau_{i,j}(A) \cd \tau_{j,k}(A) = \tau_{i,k}(A)
}
\]
\end{lemma}
\begin{proof}
This is a straight forward verification. Indeed, the transition numbers are
defined as quotients of determinants.
\end{proof}

\begin{remark}
The result of Lemma \ref{trans} indicates that it should be possible to define
a joint torsion line bundle by gluing together appropriate joint torsion
transition functions. For example, in the Hilbert space setup each of the
joint torsion transition number $\tau_{i,j}(A)$ can be extended to a
transition function over a suitable open subset $U_i \cap U_j$ of
$\cc^n$. This open subset is related to the joint essential spectrum of the
bounded operators involved. This would give rise to an interesting joint
torsion line bundle which could be the subject of future research. The aim of
the present text is however to investigate the properties of the above joint
torsion transition numbers.
\end{remark}

Let us end this section by studying the behaviour of the joint torsion
transition numbers under permutations. Thus, let $\si : \{1,\ldots,n\} \to
\{1,\ldots,n\}$ be a permutation. As in Section \ref{fredthe} we can form the
commuting tuple
\[
\si(A) = (A_{\si(1)},\ldots,A_{\si(n)})
\]
It follows from Theorem \ref{indsymm} that the commuting tuples
\[
\arr{ccc}{
\si^{-1}(i)\big( \si(A) \big) & & \si^{-1}(j)\big( \si(A) \big)
}
\]
are Fredholm. Recall in this respect that $i(A)$ and $j(A)$ are assumed to be
Fredholm. In particular, we can make sense of the joint torsion transition
number in position $(\si^{-1}(i), \si^{-1}(j))$,
\[
\tau_{\si^{-1}(i), \si^{-1}(j)}\big( \si(A) \big) \in \ff^*
\]
of the permuted tuple $\si(A)$. Before proving a symmetry property for the
joint torsion we introduce some extra maps.

Let $k \in \{1,\ldots,n\}$. We define the shift isomorphism
\[
\io_k : \{1,\ldots,n-1\} \to \{1,\ldots,n\} - \{k\}
\] 
by the formula
\[
\arr{ccc}{
\io_k : m \mapsto \fork{ccc}{
m & \T{for} & m < k \\
m + 1 & \T{for} & m \geq k
}
}
\]
Furthermore, we define the permutation $k(\si) \in \Si_{n-1}$ by the formula
\[
k(\si) := \io_{\si(k)}^{-1} \ci \si \ci \io_k : \{1,\ldots,n-1\} \to
\{1,\ldots,n-1\}
\]
Notice that the image of $\si \ci \io_k$ is contained in the set
$\{1,\ldots,n\} - \{\si(k)\}$. We then have the identity
\begin{equation}\label{eq:idpermu}
k(\si(A)) = k(\si)\big( \si(k)(A) \big)
\end{equation}
of commuting tuples.

\begin{prop}\label{jointsymm}
The joint torsion transition number in position $(i,j)$ of the commuting tuple
$A$ coincides with the joint torsion transition number in position
$(\si^{-1}(i),\si^{-1}(j))$ of the permuted tuple $\si(A)$. Thus, in formulas
we have that
\[
\tau_{i,j}(A) = \tau_{\si^{-1}(i),\si^{-1}(j)}\big( \si(A)\big)
\]
\end{prop}
\begin{proof}
It follows from \eqref{eq:idpermu} that we have an isomorphism
\[
k(\si)^{-1} : K(\si(k)(A)) \to K\big(k(\si(A))\big)
\]
of Koszul complexes. See the proof of Theorem \ref{indsymm}. This chain map
fits in a commutative diagram
\[
\begin{CD}
K\big( \si(k)(A) \big) @>A_{\si(k)}>> K\big( \si(k)(A) \big)[1]
@>\io_{\si(k)}>> 
K(A) @>\io_{\si(k)}^* \ep_{\si(k)}^*>> K\big( \si(k)(A) \big)
\\
@V{k(\si)^{-1}}VV @V{k(\si)^{-1}}VV @V{\si^{-1}}VV @V{k(\si)^{-1}}VV
\\
K\big( k(\si(A))\big) @>A_{\si(k)}>> K\big( k(\si(A)) \big)[1]
@>\io_k>> K(\si(A)) @>\io_k^* \ep_k^*>>
K\big( k(\si(A))\big)
\end{CD}
\]
where the rows are odd homotopy exact triangles and the columns are even
isomorphisms of chain complexes. In particular, we get the identities
\[
\begin{split}
T(X^A_i) 
& = \T{det}\big(H_-(\si^{-1})\big)^{-1} \ci 
T(X^{\si(A)}_{\si^{-1}(i)}) \ci \T{det}\big( H_+(\si^{-1}) \big)  \\
T(X^A_j) 
& = \T{det}\big(H_-(\si^{-1})\big)^{-1} \ci
T(X^{\si(A)}_{\si^{-1}(j)}) \ci \T{det}\big( H_+(\si^{-1}) \big)
\end{split}
\]
of torsion isomorphisms. Here the even isomorphisms
\[
\arr{ccc}{
H(\si^{-1}) : H(X^A_{\si(k)}) \to H(X^{\si(A)}_k) & & \si(k) = i,j
}
\]
are induced by the columns of the above commutative diagram. The desired
result is now a consequence of these observations and some basic properties of
determinants.
\end{proof}



\subsection{Example : Lefschetz numbers}
Let $i,j \in \{1,\ldots,n\}$ be two numbers and suppose that the commuting
tuples
\[
\arr{ccc}{
i(A) = (A_1,\ldots,\wih{A_i},\ldots,A_n)
& \T{and} &
j(A) = (A_1,\ldots,\wih{A_j},\ldots,A_n)
}
\]
are Fredholm. Furthermore, suppose that the Koszul homology of $A$ is trivial,
thus $H(A) = \{0\}$.

The Fredholm condition on $i(A)$ and $j(A)$ ensures us that the joint torsion
transition number $\tau_{i,j}(A) \in \ff^*$ is well-defined. Furthermore, the
vanishing condition on the Koszul homology of $A$ entails that the even chain
maps
\[
\arr{ccc}{
A_i : K\big( i(A) \big) \to K\big( i(A)\big) & \T{and} & A_j : K\big( j(A)
\big) \to K\big( j(A)\big)
}
\]
induces isomorphisms at the level of homology. This is a consequence of
Theorem \ref{minusex}. The next theorem gives an expression for the joint
torsion transition number in terms of quotients of determinants of these
induced isomorphisms on homology.

\begin{prop}\label{lefschetz}
The joint torsion transition number $\tau_{i,j}(A) \in \ff^*$ coincides with
the product of quotients of determinants
\[
\tau_{i,j}(A) =
\left( \frac{\E{det}\big(A_i|_{H_+(i(A))}\big)}{
\E{det}\big(A_i|_{H_-(i(A))}\big)} \right) \cd
\left( \frac{\E{det}\big(A_j|_{H_-(j(A))}\big)}{
\E{det}\big(A_j|_{H_+(j(A))}\big)} \right)
\]
Here the notation
\[
\arr{ccc}{
A_k|_{H_\pm\big(k(A)\big)} : H_{\pm}\big(k(A) \big) \to H_{\pm}\big(k(A) \big)
& & k = i,j
}
\]
refers to the restriction of $A_k$ to the positive or negative part of the
Koszul homology of the commuting tuple $k(A)$.
\end{prop}
\begin{proof}
The $\zz_2$-graded vector space $H(X^A_i)$ associated with the odd homotopy
exact triangle $X^A_i$ is given by the components
\[
\arr{ccc}{
H_+(X^A_i) = H\big( i(A) \big)
& \T{and} &
H_-(X^A_i) = H\big( i(A) \big)[1]
}
\]
Here the notation "$[1]$" refers to the operation of taking the opposite
grading. Furthermore, the odd exact endomorphism
\[
H(\al_i) \in \T{End}_-(H(X^A_i))
\]
associated with the odd homotopy exact triangle $X^A_i$ is given by the
matrices
\[
\begin{split}
H_+(\al_i) & = \mat{cc}{
0 & 0 \\
A_i & 0
} : H\big(i(A)\big) \to H\big(i(A) \big)[1] \q \T{and} \\
H_-(\al_-) & = \mat{cc}{
0 & 0 \\
A_i & 0
} : H\big(i(A)\big)[1] \to H\big(i(A)\big)
\end{split}
\]
See \eqref{eq:oddhomtri}. The torsion isomorphism of the odd homotopy exact
triangle $X^A_i$ is therefore given by the determinant
\[
T(X^A_i)
= \T{det}\mat{cc}{
0 & A_i^{-1} \\
A_i & 0
} : \T{det}\big(H(i(A))\big) \to \T{det}\big(H(i(A))[1]\big)
\]
By an application of the isomorphisms in \eqref{eq:isomdethom} we get that
$T(X^A_i)$ identifies with the quotient of determinants
\[
T_i(A) = 
(-1)^{\mu_i(A)} \frac{\T{det}\big(A_i|_{H_+(i(A))}\big)}{
  \T{det}\big(A_i|_{H_-(i(A))}\big)} \in \ff^*
\]
Here the exponent $\mu_i(A) \in \nn \cup \{0\}$ can be found in Definition
\ref{defjoint}. The result of the theorem now follows by noting that the same
calculations can be applied when $i$ is replaced by $j$.
\end{proof}

The product of quotients of determinants obtained in Theorem \ref{lefschetz}
is referred to by R. Carey and J. Pincus as a "Lefschetz number". See
\cite[\S $4$ p. $289$]{carpincI}. In the case of a pair of commuting Fredholm
operators with vanishing Koszul homology their associated Lefschetz number has
been studied in \cite{carpincII}. In particular, the relation between these
numbers and the second algebraic $K$-group was clarified. One of the
motivations for introducing a joint torsion invariant is to generalize the
notion of a Lefschetz number to the more general setup where the vanishing
condition on Koszul homology is removed. See \cite[\S $1$
p. $128$]{carpincIII}. The result of Theorem \ref{lefschetz} therefore gives
some justification for our definition of a joint torsion transition number.



\subsection{Example : Toeplitz operators over the polydisc}
Let $n \in \nn$, let $\B T^n \su \cc^n$ be the $n$-dimensional torus and let
$\B D^n \su \cc^n$ be the polydisc. The interior of $\B D^n$ will be denoted
by $U^n$.

We let $L^2(\B T^n)$ denote the Hilbert space of square integrable functions
on the torus. The continuous functions on the torus act by pointwise
multiplication on the $L^2$-functions. Thus, we have an algebra homomorphism
\[
\arr{ccc}{
m : C(\B T^n) \to \C L\big(L^2(\B T^n)\big)
& & m(f)(g) = f \cd g
}
\]
Here $\C L\big( L^2(\B T^n) \big)$ denotes the bounded operators on the
Hilbert space $L^2(\B T^n)$.

We let $\C A(U^n)$ denote the polydisc algebra. Thus, $\C A(U^n)$ consists of
the continuous functions on the polydisc $\B D^n$ which restrict to
holomorphic functions on the interior of the poly-disc, $U^n$. The algebraic
operations on $\C A(U^n)$ are the pointwise versions of sum, product and scalar
multiplication. See \cite{rudin}.

We let $H^2(U^n) \su L^2(\B T^n)$ denote the Hardy-space over the
poly-disc. This is the smallest sub-Hilbert space of the $L^2$-functions
generated by the continuous functions which extend to elements of the polydisc
algebra. The orthogonal projection onto Hardy-space will be denoted by $P \in
\C L\big(L^2(\B T^n)\big)$.

\begin{definition}
By the \emph{Toeplitz operator} associated to a continuous function $f \in
C(\B T^n)$ we will understand the restriction of its multiplication operator
to Hardy space. The corresponding Toeplitz operator will be denoted by $T_f
\in \C L\big(H^2(U^n)\big)$. Thus, by definition
\[
T_f = P m(f) P : H^2(U^n) \to H^2(U^n)
\]
\end{definition}

We let $z_1,\ldots,z_n : \B T^n \to \cc$ denote the coordinate functions on
the torus.

Now, let $f \in \C A(U^n)$ be an \emph{invertible} element of the poly-disc
algebra. By a slight abuse of notation we will also let $f \in C(\B T^n)$
denote the restriction of $f$ to the $n$-torus. We will look at the commuting
tuple of Toeplitz operators
\[
T_\al := (T_f, T_{z_1 - \al_1},\ldots,T_{z_n - \al_n})
\]
where $\al = (\al_1,\ldots,\al_n) \in U^n$ is an element of the interior of
the poly-disc. Notice that the invertibility condition on $f$ implies that the
Koszul homology groups
\[
\arr{ccc}{
H(T_\al) = \{0\} = H\big( i(T_\al)\big) & \q & i \neq 1
}
\]
vanish. Furthermore, it is not hard to show, that the Koszul homology of
$1(T_\al)$ is given by the components
\[
\arr{ccc}{
H_+\big( 1(T_\al)\big) 
= \C H / \big( m(z_1 - \al_1) \C H \plp m(z_n - \al_n) \C H \big)
& \T{and} & H_-\big( 1(T_\al)\big) = \{0\}
}
\]
Here $\C H = H^2(U^n)$ is notation for the Hardy-space over the
poly-disc. The dimension of $H_+\big( 1(T_\al)\big)$ is equal to one. One way
of seeing this, is to think of the case where $\al = (0,\ldots,0)$. In this
situation, we get that $H_+\big(1(T_\al)\big)$ is spanned by the constant
function equal to one. To obtain the result in the general setting it now
suffices to use the homotopy invariance of the Fredholm tuple index. See
\cite{curto}. The dimension of $H_+\big( 1(T_\al) \big)$ can also be found by
a direct calculation.

The finite dimensionality of the quotient space implies that the subspace 
\[
m(z_1- \al_n) \C H \plp m(z_n - \al_n) \C H \su \C H
\]
is closed. See \cite[\S $6$ Theorem $2$]{curto}. These observations allow us
to compute the joint torsion transition numbers of our commuting tuple of
Toeplitz operators.

\begin{prop}\label{jointcauchy}
Let $1 \leq i < j \leq n$ be two numbers. The joint torsion transition numbers
of the commuting tuple $T_\al$ are given by
\[
\tau_{i,j}(T_\al) = \fork{ccc}{
f(\al) & \T{for} & i = 1 \\
1 & \T{for} & i \neq 1
}
\]
\end{prop}
\begin{proof}
We will only consider the case where $i = 1$. By an application of Theorem
\ref{lefschetz} we get that the joint torsion transition number
$\tau_{1,j}(T_\al)$ is given by the determinant of
\[
T_f : H_+\big(1(T_\al)\big) \to H_+\big(1(T_\al)\big)
\]
But this isomorphism coincides with the multiplication by the non-zero
constant $f(\al) \in \cc^*$. This proves the theorem.
\end{proof}

Remark that the above calculation yields an interesting link between the joint
torsion transition numbers and the Cauchy integral formula. Indeed, we have
that
\[
\frac{1}{(2 \pi i)^n} \oint \frac{f(z)}{(z_1 - \al_1)\clc (z_n - \al_n)} dz
= f(\al) = \tau_{1,j}(T_f,T_{z_1 - \al_1},\ldots,T_{z_n - \al_n})
\]
Furthermore, it shows that the joint torsion transition numbers are far from
being homotopy invariant: We could reproduce all the values of the invertible
holomorphic function $f$ on the interior of the polydisc.





\section{Comparison of vertical and horizontal torsion isomorphisms}
In this section we study the basic question: When do the torsion isomorphisms
of two anti-commuting odd exact endomorphisms on a finite dimensional
$\zz_2$-graded vector space coincide? We shall see that this is the case when
the odd exact endomorphisms come from odd homotopy exact triangles which fit
together as rows and columns in a larger diagram satisfying some
(anti)-commutativity conditions. Later on we will apply this comparison
theorem to prove some more involved results on our joint torsion transition
numbers. In the first subsection we establish the framework of odd homotopy
exact bitriangles. The definitions which we give are minimal in the sense that
they are tailored to handle the applications which we have in mind. In the
second subsection we compare the torsion isomorphisms of general
anti-commuting odd exact endomorphisms. We think that the determinants
appearing in Theorem \ref{compztwo} are an obstruction for such a pair of
torsion isomorphisms to agree in general. In the last subsection we prove that
the vertical and horizontal torsion isomorphisms of an odd homotopy exact
bitriangle agree. In particular, the obstruction mentioned above vanishes in
this situation. The result of this section might be interpretted as the
associativity of a determinant functor on a certain triangulated category. See
\cite[Definition $3.1$]{breuning} and \cite[Definition $1.3.4$]{murotonkwit}.

\subsection{Odd homotopy exact bitriangles}\label{bitri}
Let us consider three odd triangles of $\zz_2$-graded chain complexes,
$X_{*1}, X_{*2}$ and $X_{*3}$. Thus, to fix the notation we have the following
picture
\[
\begin{CD}
X_{*j} \, \, : \, \,
X_{1j} @>{v_{1j}}>> X_{2j} @>{v_{2j}}>> X_{3j}
@>{v_{3j}}>> X_{1j}
\end{CD}
\]
for each $j \in \{1,2,3\}$. We will think of these three odd triangles as
being "vertical". Let us then look at three odd triangles of chain complexes
$X_{1*}, X_{2*}$ and $X_{3*}$. Thus, again, to fix the notation we have the
following picture
\[
\begin{CD}
X_{i*} \, \, : \, \,
X_{i1} @>{h_{i1}}>> X_{i2} @>{h_{i2}}>> X_{i3}
@>{h_{i3}}>> X_{i1}
\end{CD}
\]
for each $i \in \{1,2,3\}$. We will think of these three triangles as being
"horizontal".

The differential of the $\zz_2$-graded chain complex $X_{i,j}$ will be denoted
by
\[
d_{i,j} : X_{i,j} \to X_{i,j}
\]

\begin{definition}\label{oddbit}
We will say that the data $(X,v,h)$ is an odd \emph{bitriangle} of chain
complexes when the diagrams of odd chain maps
\[
\begin{CD}
X_{i,j} @>{h_{i,j}}>> X_{i,(j+1)} \\
@V{v_{i,j}}VV @V{v_{i,(j+1)}}VV \\
X_{(i+1),j} @>h_{(i+1),j}>> X_{(i+1),(j+1)}
\end{CD}
\] 
are anti-commutative up to chain homotopy for all $i,j \in \{1,2,3\}$. Thus,
we assume the existence of odd linear maps
\[
s_{i,j} : X_{i,j} \to X_{(i+1),(j+1)}
\]
such that
\[
d_{(i+1),(j+1)} s_{i,j} + s_{i,j} d_{i,j} = h_{(i+1),j} v_{i,j} + v_{i,(j+1)}
h_{i,j}
\]
for all $i,j \in \{1,2,3\}$. Here we are calculating with the indices modulo
three.
\end{definition}

Now, suppose that $(X,v,h)$ is an odd bitriangle of chain
complexes. Furthermore, suppose that each of the vertical odd triangles and
each of the horizontal odd triangles are homotopy exact. See Definition
\ref{homoex}. We will use the notation
\[
\arr{ccc}{
t_{i,j} : X_{(i+1),j} \to X_{i,j} & \T{and} & 
r_{i,j} : X_{i,(j+1)} \to X_{i,j}
}
\]
for the corresponding vertical and horizontal homotopies.

We will often suppress the bi-indices of the linear maps involved.

\begin{lemma}
The even linear maps
\[
\arr{ccc}{
vs + sv + t h + h t : X_{i,j} \to X_{(i-1),(j+1)}
& \T{and} & hs + sh + r v + v r : X_{i,j} \to X_{(i+1),(j-1)}
}
\]
are even chain maps for all $i,j \in \{1,2,3\}$.
\end{lemma}
\begin{proof}
We will only consider the first of the two even maps, since the other case is
completely similar. The chain map property follows from the calculation,
\[
\begin{split}
d(vs + sv + t h + h t)
& = -vds + (vh + hv - sd)v + (v^2 - td)h - hd t \\
& = -v(vh + hv -sd) + (vh + hv)v + svd + v^2 h + thd -h(v^2 - td) \\
& = (vs + sv + t h + h t)d
\end{split}
\]
\end{proof}

We will now introduce the notion of homotopy exactness for odd bitriangles.

\begin{definition}\label{defhomoex}
We say that the odd bitriangle $(X,v,h)$ with homotopy exact rows and columns
is \emph{homotopy exact} when the even chain maps
\[
\arr{ccc}{
vs + sv + t h + h t : X_{i,j} \to X_{i-1,j+1}
& \T{and} & hs + sh + r v + v r : X_{i,j} \to X_{i+1,j-1}
}
\]
vanish at the level of homology for each $i,j \in \{1,2,3\}$.
\end{definition}

Suppose that $(X,v,h)$ is homotopy exact and that the homology of the
$\zz_2$-graded chain complex $X_{i,j}$ is finite dimensional for each $i,j \in
\{1,2,3\}$. We thus have three "horizontal" torsion isomorphisms
\[
\arr{ccc}{
T(X_{i*}) : \T{det}\big(H_+(X_{i*})\big) \to
\T{det}\big(H_-(X_{i*})\big)
& & i \in \{1,2,3\}
}
\]
and three "vertical" torsion isomorphisms
\[
\arr{ccc}{
T(X_{*j}) : \T{det}\big(H_+(X_{*j})\big) \to
\T{det}\big(H_-(X_{*j})\big) & & j \in \{1,2,3\}
}
\]
It is the goal of the next sections to compare these isomorphisms. Later on we
will use this comparison result to prove some algebraic properties of the
joint torsion transition numbers. The proof of these algebraic properties is
the main achievement of this paper. We will begin with an analysis of the
general $\zz_2$-graded situation.



\subsection{Comparison of odd exact endomorphisms}
Let us consider a $\zz_2$-graded vector space $V = V_+ \op V_-$ of finite
dimension. We suppose that $V$ comes equipped with two odd exact endomorphisms
\[
\arr{ccc}{
v : V \to V & \T{and} & h : V \to V 
}
\]
Thus, we have two torsion isomorphisms
\[
\begin{split}
T(v) & = \T{det}(v_+ + v^\da_-) : \T{det}(V_+) \to \T{det}(V_-) \q \T{and} \\
T(h) & = \T{det}(h_+ + h^\da_-) : \T{det}(V_+) \to \T{det}(V_-)
\end{split}
\]
between the same one-dimensional vector spaces. These two torsion isomorphisms
can then be compared by looking at the invertible number
\begin{equation}\label{eq:compbasic}
T(h)^{-1} \ci T(v)
= \T{det}(h_+^\da + h_-) \ci \T{det}(v_+ + v_-^\da) 
\in \ff^*
\end{equation}
This quantity will be the subject of our attention in this section and we will
refer to it as the \emph{comparison number}.

Let us assume that $v$ and $h$ anti-commute. Thus, we have the identity
\[
vh + hv = 0 : V \to V
\]
of linear maps.

We will use the notation
\[
\arr{ccccc}{
K^v := \T{Ker}(v) & & K^h := \T{Ker}(h) & & K^{vh} := \T{Ker}(vh) =
\T{Ker}(hv)
}
\]
for the various kernels. Furthermore, we will use the notation
\[
I := \T{Im}(vh) = \T{Im}(hv)
\]
for the image of the composition of the odd exact endomorphisms. Notice that
all these subspaces are $\zz_2$-graded by the induced grading.

The next lemma follows from the anti-commutativity and exactness assumptions
on our odd endomorphisms. We will use it to define a relevant algebraic
decomposition of our $\zz_2$-graded vector space.

\begin{lemma}\label{oddisobasic}
The odd exact endomorphisms $v,h \in \E{End}_-(V)$ induce odd isomorphisms
\[
v \,\, \T{ and }\,\, h : K^{vh}/( K^v + K^h ) \to (K^v \cap K^h)/I
\]
and odd isomorphisms
\[
\arr{ccc}{
v : V/K^{vh} \to K^v/(K^v \cap K^h) & \T{and} &
h : V/K^{vh} \to K^h/(K^v \cap K^h)
}
\]
\end{lemma}

Let us choose subspaces $R_+ \su K_+^{vh}$ and $R_- \su K_-^{vh}$ such that
the quotient maps
\[
\arr{ccc}{
R_+ \to K^{vh}_+/(K^v_+ + K^h_+) & \T{and} &
R_- \to K^{vh}_-/(K^v_- + K^h_-)
}
\]
become isomorphisms. We let $R := R_+ \op R_- \su V$ denote the $\zz_2$-graded
subspace given by these two components.

Likewise, let us choose subspaces $S_+ \su V_+$ and $S_- \su V_-$ such that
the quotient maps
\[
\arr{ccc}{
S_+ \to V_+/K^{vh}_+ & \T{and} & S_- \to V_-/K^{vh}_-
}
\]
become isomorphisms. We let $S := S_+ \op S_- \su V$ denote the $\zz_2$-graded
subspace given by these two components.

We will then use the notation
\[
\arr{ccc}{
C^v := \T{Im}(v|_R) & \T{and} & C^h := \T{Im}(h|_R) \\
L := \T{Im}(v|_S) & \T{and} & M := \T{Im}(h|_S)
}
\]
for the various images associated with these subspaces. Notice that all these
images are $\zz_2$-graded by the induced grading coming from $V$. We then have
two different algebraic decompositions of the vector space $V$.

\begin{lemma}\label{algdecompbas}
The vector space sum induces even isomorphisms
\[
\begin{split}
& I \op C^v \op L \op M \op R \op S \cong V \q \T{and} \\
& I \op C^h \op L \op M \op R \op S \cong V
\end{split}
\]
of $\zz_2$-graded vector spaces.
\end{lemma}
\begin{proof}
This follows immediately from the definitions of the subspaces and Lemma
\ref{oddisobasic}.
\end{proof}

We will use the algebraic decompositions of Lemma \ref{algdecompbas} to define
convenient pseudo-inverses
\[
\arr{ccc}{
h_+^\da : V_- \to V_+ & \T{and} & v_-^\da : V_+ \to V_-
}
\]
Indeed, we can choose the pseudo-inverse $h_+^\da$ such that
\[
\arr{ccc}{
\T{Ker}(h_+^\da) = L_- + R_- + S_- & \T{and} & 
\T{Im}(h_+^\da) = L_+ + R_+ + S_+
}
\]
This pseudo-inverse is given explicitly on $\T{Im}(h_+) = I_- + C^h_- + M_-$
by the formula
\[
\arr{ccc}{
h_+^\da\big( v_+h_-(x_1) + h_+(x_2) + h_+(x_3)\big)
= -v_-(x_1) + x_2 + x_3 
& & x_1 \in S_- \, , \, x_2 \in R_+ \, , \, x_3 \in S_+
}
\]
Likewise, we can choose the pseudo-inverse $v_-^\da$ such that
\[
\arr{ccc}{
\T{Ker}(v_-^\da) = M_+ + R_+ + S_+ & \T{and} & 
\T{Im}(v_-^\da) = M_- + R_- + S_-
}
\]
This pseudo-inverse is given explicitly on $\T{Im}(v_-) = I_+ + C^v_+ + L_+$
by the formula
\[
\arr{ccc}{
v_-^\da\big( v_-h_+(x_1) + v_-(x_2) + v_-(x_3)\big)
= h_+(x_1) + x_2 + x_3
& & x_1 \in S_+ \, , \, x_2 \in R_- \, , \, x_3 \in S_-
}
\]
These explicit choices of pseudo-inverses will allow us to find a simpler
formula for the comparison number \eqref{eq:compbasic}. In the next lemma we
analyze the behaviour of the isomorphism
\[
(h_- + h_+^\da) \ci (v_+ + v_-^\da) : V_+ \to V_+
\]
on two invariant subspaces.

\begin{lemma}\label{invardettriv}
The subspaces $I_+ + S_+ \su V_+$ and $L_+ + M_+ \su V_+$ are invariant under
the isomorphism
\[
(h_- + h_+^\da) \ci (v_+ + v_-^\da) : V_+ \to V_+
\]
Furthermore, the determinants of the restrictions
\[
\begin{split}
& \E{det}\big( (h_- + h_+^\da)(v_+ + v_-^\da) |_{I_+ + S_+} \big) = 1 \q \T{and}
\\
& \E{det}\big( (h_- + h_+^\da)(v_+ + v_-^\da) |_{L_+ + M_+} \big) = 1
\end{split}
\]
are both trivial.
\end{lemma}
\begin{proof}
Let us look at an arbitrary element
\[
\arr{ccc}{
v_-h_+(x) + y \in I_+ + S_+ & & x,y \in S_+
}
\]
We can then calculate as follows
\[
(h_- + h_+^\da)(v_+ + v_-^\da)\big((v_-h_+)(x) + y \big)
= (h_- + h_+^\da)(h_+(x) + v_+(y) )
= x + (h_-v_+)(y)
\]
This proves that $I_+ + S_+$ is an invariant subspace for the automorphism
$(h_- + h_+^\da)(v_+ + v_-^\da) \in \T{Aut}(V_+)$. Furthermore, we see that
the restriction of this automorphism to the invariant subspace in question can
be written on the matrix form
\[
(h_- + h_+^\da)(v_+ + v_-^\da) = \mat{cc}{
0 & -f \\
f^{-1} & 0
} : I_+ \op S_+ \to I_+ \op S_+
\]
Here the letter $f$ refers to the isomorphism
\[
f := v_- h_+ : S_+ \to I_+
\]
In particular, the determinant of the restriction to this subspace is
trivial. This proves the first part of the lemma.

For the second part, we look at an arbitrary element
\[
\arr{ccc}{
v_-(x) + h_-(y) \in L_+ + M_+ & & x,y \in S_-
}
\]
in the subspace $L_+ + M_+$. We can then calculate as follows
\[
(h_- + h_+^\da)(v_+ + v_-^\da)\big(v_-(x) + h_-(y)\big)
= (h_- + h_+^\da)\big( x + (v_+h_-)(y)\big)
= h_-(x) - v_-(y)
\]
This proves that the subspace $L_+ + M_+ \su V_+$ is invariant under the
automorphism $(h_- + h_+^\da)(v_+ + v_-^\da) \in \T{Aut}(V_+)$. Furthermore,
we see that the restriction can be written on the matrix form
\[
(h_- + h_+^\da)(v_+ + v_-^\da)
= \mat{cc}{
0 & - g \\
g^{-1} & 0
} : L_+ \op M_+ \to L_+ \op M_+
\]
Here $g : M_+ \to L_+$ is the isomorphism given by
\[
\arr{ccc}{
g : h_-(x) \mapsto v_-(x) & & x \in S_-
}
\]
In particular, the determinant of the restriction to this subspace is
trivial. These observations prove the second part of the lemma.
\end{proof}

Let us use notation $P_R : V \to V$ for the unique idempotent with
\[
\arr{ccc}{
\T{Im}(P_R) = R & \T{and} & \T{Ker}(P_R) = K^v + K^h + S
}
\]
Furthermore, we will use the notation $P_{C^v} : V \to V$ for the unique
idempotent with
\[
\arr{ccc}{
\T{Im}(P_{C^v}) = C^v & \T{and} & \T{Ker}(P_{C^v}) = I + L + M + R + S
}
\]

\begin{lemma}\label{comparprodbas}
The comparison number is given by the product of determinants
\[
T(h)^{-1} \ci T(v) = 
\E{det}(P_R h^{\da}_+ v_+|_{R_+}) \cd
\E{det}(P_{C^v} h_- v_-^\da|_{C^v_+})
\in \ff^*
\]
of restrictions to the subspaces $R_+$ and $C^v_+$.
\end{lemma}
\begin{proof}
We consider the restrictions
\[
\begin{split}
(h^\da_+ + h_-)(v_+ + v_-^\da)|_{R_+} & = 
h^\da_+ v_+|_{R_+} : R_+ \to V_+ \q \T{and} \\
(h^\da_+ + h_-)(v_+ + v_-^\da)|_{C^v_+} & =
h_- v_-^\da|_{C^v_+} : C^v_+ \to V_+ 
\end{split}
\]
It can be verified that we have the inclusions
\[
\arr{ccc}{
\T{Im}(h^\da_+ v_+|_{R_+}) \su R_+ + L_+
& \T{and} & \T{Im}(h_- v_-^\da|_{C^v_+}) \su C^v_+ + I_+
}
\]
for the images of these restrictions. In particular, we can write the
automorphism $(h_+^\da + h_-)(v_+ + v_-^\da) \in \T{Aut}(V_+)$ as an upper
triangular matrix with respect to the decomposition
\[
V_+ \cong (I_+ + S_+) \op (L_+ + M_+) \op R_+ \op C_+^v
\]
The result of the lemma is now a consequence of Lemma \ref{invardettriv} and
basic properties of determinants.
\end{proof}

Let us use the letters $\xi$ and $\eta$ for the odd isomorphisms
\[
\begin{split}
& \xi := v : K^{vh}/(K^v + K^h) \to (K^v \cap K^h)/I \q \T{and} \\
& \eta := h : K^{vh}/(K^v + K^h) \to (K^v \cap K^h)/I
\end{split}
\]
induced on quotient spaces by the odd exact endomorphisms $v$ and $h$. See
Lemma \ref{oddisobasic}.

\begin{prop}\label{compztwo}
The comparison number is given by the product of determinants
\[
T(h)^{-1} \ci T(v) =
\E{det}(\eta^{-1}_+ \xi_+) \cd \E{det}(\eta_- \xi_-^{-1}) \in \ff^*
\]
on quotient spaces.
\end{prop}
\begin{proof}
This follows immediately from Lemma \ref{comparprodbas}.
\end{proof}



\subsection{The comparison theorem}\label{compathe}
Let $(X,v,h)$ be an odd homotopy exact bitriangle of chain complexes.

We define the "diagonal" $\zz_2$-graded chain complexes as the direct sums of
$\zz_2$-graded chain complexes
\begin{equation}\label{eq:decopiece}
\begin{split}
D_1 & := X_{31} \op X_{22} \op X_{13} \\
D_2 & := X_{11} \op X_{32} \op X_{23} \\
D_3 & := X_{21} \op X_{12} \op X_{33}
\end{split}
\end{equation}
The $\zz_2$-graded homology of the $\zz_2$-graded chain complex $D_k$ is
denoted by $H(D_k)$ for each $k \in \{1,2,3\}$. We let $H(X)$ denote the
$\zz_2$-graded vector space given by the direct sum of homology groups
\begin{equation}\label{eq:decodiag}
H(X) := H(D_1) \op H(D_2) \op H(D_3)
\end{equation}
We will refer to $H(X)$ as the \emph{homology of the bitriangle} $X$. Let us
assume that the homology group $H(X)$ is a finite dimensional vector space.

The vertical and horizontal odd chain maps of the odd bitriangle $X$ induce
odd endomorphisms
\[
\arr{ccc}{
v : H(X) \to H(X) & \T{and} & h : H(X) \to H(X)
}
\]
of the homology of $X$. These odd endomorphisms are exact since the columns
and rows of the homotopy exact bitriangle are homotopy exact. Furthermore, the
anti-commute since the odd chain maps anti-commute up to chain homotopy. We
are thus in the situation of the last section. In particular, we get the
identity
\[
T(h)^{-1} \ci T(v) = \T{det}(\eta_+^{-1} \xi_+) \cd \T{det}(\eta_- \xi_-^{-1})
\in \ff^*
\]
for the associated comparison number. See Theorem \ref{compztwo}. Here we
recall that $\xi$ and $\eta$ denote the odd isomorphisms
\[
\xi \T{ and } \eta : K^{vh}/(K^v + K^h) \to (K^v \cap K^h)/I
\]
of quotient spaces induced by the odd maps $v,h \in \T{End}_-(H(X))$. The
purpose of this section is to show that the comparison number of an odd
homotopy exact bitriangle is trivial. This will allow us to compare the
torsion isomorphisms of the columns and rows of the homotopy exact bitriangle
$(X,v,h)$. See the end of Section \ref{bitri}. We will start by obtaining a
description of the space $K^{vh}/(K^v + K^h)$ as the quotient of a homology
group.

Let $k \in \{1,2,3\}$. The vertical and horizontal maps induce odd linear maps
\[
\arr{ccc}{
v : H(D_k) \to H(D_{k+1}) & \T{and} & h : H(D_k) \to H(D_{k+1})
}
\]
on the diagonal homology groups. These induced vertical and horizontal odd
maps anti-commute and have trivial squares. We can thus define a differential
\[
\arr{ccc}{
\de_k : H(D_k) \to H(D_{k+1}) & & \de_k = v - h
}
\]
The corresponding homology groups will be denoted by
\[
H_\de\big(H(D_k)\big) := \T{Ker}(\de_k)/ \T{Im}(\de_{k-1})
\]
Notice that these homology groups are $\zz_2$-graded vector spaces with the
induced grading coming from the $\zz_2$-graded homology groups $H(D_k)$.

Let $i,j \in \{1,2,3\}$ be two numbers with sum congruent to $k$ modulo $3$,
$i+j \co_3 k$. We let
\[
P_{i,j} : H(D_k) \to H(X_{i,j})
\]
denote the map which is induced by the even projection onto the component
$X_{i,j}$ relative to the decomposition \eqref{eq:decopiece}.

We then have a well-defined even homomorphism
\begin{equation}
\label{eq:homocruc}
P_{i,j} : H_\de\big( H(D_k)\big) 
\to P_{i,j}\big(K^{vh}/(K^v + K^h)\big)
\end{equation}
of $\zz_2$-graded vector spaces.

\begin{lemma}\label{surjcruc}
The even linear map
\[
P_{i,j} : 
H_\de\big( H(D_k)\big) \to P_{i,j} \big( K^{vh}/(K^v + K^h)\big)
\]
is surjective.
\end{lemma}
\begin{proof}
We will apply the notation of Section \ref{bitri} for the various chain
homotopies.

Suppose that $x \in X_{i,j} = P_{i,j}D_k$ is a cycle and that we have an
element $y \in X_{(i+1),(j+1)}$ such that
\[
(vh)(x) = d(y)
\]

We notice that
\[
(hv)(x) = (ds + sd - vh)(x) = d(sx - y)
\]

Define two elements $z, w \in D_k$ by the formulae
\[
\arr{ccc}{
z := (th)(x) + v(y)
& \T{and} & w := (rv)(x) + h(sx - y)
}
\]
Notice that $P_{(i-1),(j+1)}z = z$ and that $P_{(i+1),(j-1)}w =
w$. Furthermore, from the proof of Lemma \ref{exhomol} we know that
\[
d(z) = 0 = d(w)
\]
and that we have the identities
\begin{equation}\label{eq:horvertI}
\arr{ccc}{
v[z] = h[x] & \T{and} & h[w] = v[x]
}
\end{equation}
in the homology group $H(D_{k+1})$.

We shall now see that the value of the horizontal map at the class $[z]
\in H(D_k)$ agree with the value of the vertical map at the class $[w] \in
H(D_k)$. This follows from the calculation
\begin{equation}\label{eq:horvertII}
\begin{split}
h[z] 
& = \big[ (hth)(x) + (hv)(y)\big] \\
& = \big[-h(vs + sv + ht)(x) + (sd-vh)(y)\big] \\
& = \big[(vh - sd)(sx) -(hsv)(x) - (rdt)(x) + (svh)(x) - (vh)(y)\big] \\
& = \big[(vh - sd)(sx) - (hsv)(x) - (rv^2)(x) + s(ds - hv)(x) - (vh)(y)\big] \\
& = \big[(vh)(sx - y) - (hs + rv +sh)(vx)\big] \\
& = \big[(vh)(sx-y) + (vrv)(x)\big] \\
& = v[w]
\end{split}
\end{equation}
Here we are using several times that $(X,v,h)$ is an odd homotopy exact
bitriangle. In particular, we use that the chain maps
\[
th + ht + vs + sv \, \T{ and }\, sh + hs + rv + vr : X \to X
\]
vanish at the level of homology. See Definition \ref{defhomoex}.

It follows from the identities \eqref{eq:horvertI} and \eqref{eq:horvertII}
that the sum of classes
\[
[x] + [z] + [w] \in \T{Ker}(\de_k)
\]
lies in the kernel of the differential $\de_k = v - h : H(D_k) \to
H(D_{k+1})$. This proves the lemma since $P_{i,j}([x] + [z] + [w]) = [x]$.
\end{proof}

The last lemma will allow us to give a very simple description of the even
isomorphism
\begin{equation}\label{eq:isoisoiso}
\eta^{-1} \xi : K^{vh}/(K^v + K^h) \to K^{vh}/(K^v + K^h)
\end{equation}
Indeed, let us look at some element $y \in P_{i,j}K^{vh}/(K^v + K^h)$. It
follows from Lemma \ref{surjcruc} that we can find an element $x \in
H_\de(H(D_k))$ with $y = P_{i,j}x$. In particular, we get that
\begin{equation}\label{eq:simpdiag}
\begin{split}
(\eta^{-1} \xi)(P_{i,j}x) 
& = (\eta^{-1} P_{i+1,j}\xi)(x) \\
& = (\eta^{-1} P_{i+1,j} \eta)(x) \\
& = P_{i+1,j-1}(x)
\end{split}
\end{equation}
Thus the isomorphism $\eta^{-1}\xi$ is nothing but a shift along the diagonal
in the direction down and left.

Remark that the isomorphism in \eqref{eq:isoisoiso} induces even isomorphisms
\begin{equation}\label{eq:isomaldiag}
P_{i,j}K^{vh}/(K^v + K^h) \cong P_{i-1,j+1}K^{vh}/(K^v + K^h) \cong
P_{i+1,j-1}K^{vh}/(K^v + K^h)
\end{equation}
of $\zz_2$-graded vector spaces.

\begin{lemma}\label{urgtrivdet}
The determinants
\[
\E{det}(\eta^{-1}_+ \xi_+) = 1 = \E{det}(\eta_- \xi_-^{-1})
\]
are both trivial.
\end{lemma}
\begin{proof}
We will only consider the determinant $\T{det}(\eta^{-1}_+ \xi_+) \in
\ff^*$. The proof is similar for the other determinant since $\T{det}(\eta_-
\xi_-^{-1}) = \T{det}(\xi_-^{-1} \eta_-)$.

It follows from the description of the even isomorphism
\[
\eta^{-1} \xi : K^{vh}/(K^v + K^h) \to K^{vh}/(K^v + K^h)
\]
given in \eqref{eq:simpdiag} that the determinant in question is a product of
signs of cyclic permutations. Indeed, we have that
\[
\T{det}(\eta_+^{-1} \xi_+) = (-1)^{d(31)_+\big(d(22)_+ + d(13)_+\big) +
d(11)_+\big(d(32)_+ + d(23)_+\big) + d(21)_+\big(d(12)_+ + d(33)_+\big)}
\]
where the non-negative numbers $d(ij)_+ \in \nn \cup\{0\}$ are given by the
dimensions of quotient spaces
\[
d(ij)_+ = \T{dim}\Big( P_{ij}\big(K^{vh}_+/(K^v_+ + K^h_+)\big) \Big)
\]
The result of the lemma is now a consequence of the isomorphisms in
\eqref{eq:isomaldiag}.
\end{proof}

The next theorem follows from Lemma \ref{urgtrivdet} and Lemma
\ref{compztwo}. The result is important for our later investigation of the
algebraic properties of the joint torsion transition numbers.

\begin{prop}\label{compisomgood}
The vertical and the horizontal torsion isomorphism
\[
\begin{split}
& T(v) : \E{det}\big( H_+(X) \big) \to \E{det}\big(H_-(X) \big) \q \T{and} \\
& T(h) : \E{det}\big( H_+(X) \big) \to \E{det}\big(H_-(X) \big)
\end{split}
\]
of the odd homotopy exact bitriangle $(X,v,h)$ agree.
\end{prop}

Let us spell out the implications of Theorem \ref{compisomgood} for the
torsion isomorphisms of the columns and rows of the odd homotopy exact
bitriangle $X$. 

To this end, we introduce the $\zz_2$-graded homology groups $H(X_v)$ and
$H(X_h)$ which are defined as the direct sums
\[
\arr{ccc}{
H(X_v) = H(X_{*1}) \op H(X_{*2}) \op H(X_{*3}) & \T{and} &
H(X_h) = H(X_{1*}) \op H(X_{2*}) \op H(X_{3*})
}
\]
of $\zz_2$-graded homology groups.

The tensor product of the torsion isomorphisms associated with the columns
yields an isomorphism
\[
T(X_v) := T(X_{*1}) \ot T(X_{*2}) \ot T(X_{*3}) 
: \T{det}(H_+(X_v)) \to \T{det}(H_-(X_v))
\]
Likewise, the tensor product of the torsion isomorphisms associated with the
rows yields an isomorphism
\[
T(X_h) := T(X_{1*}) \ot T(X_{2*}) \ot T(X_{3*}) 
: \T{det}(H_+(X_h)) \to \T{det}(H_-(X_h))
\]
Finally, we have a permutation isomorphism
\[
\te = \mat{cc}{
\te_+ & 0 \\
0 & \te_- 
} : H_+(X_v) \op H_-(X_v) \to H_+(X_h) \op H_-(X_h)
\]
The next corollary is then a reformulation of Theorem \ref{compisomgood}.

\begin{corollary}\label{compcor}
The torsion isomorphisms of the columns and rows of the odd homotopy exact
bitriangle $X$ agree up to conjugation by the determinant of the permutation
matrix $\te$. Thus, we have the identity
\[
T(X_v) = \E{det}(\te_-)^{-1} \ci T(X_h) \ci \E{det}(\te_+)
: \E{det}(H_+(X_v)) \to \E{det}(H_-(X_v))
\]
of vertical and horizontal torsion isomorphisms.
\end{corollary}





\section{Algebraic properties}
In this section we will prove the triviality result and the multiplicativity
result for the joint torsion transition numbers. These results were advertised
in the introduction to the paper. Our main tool is the comparison theorem for
vertical and horizontal torsion isomorphisms arising from an odd homotopy
exact bitriangle. See Section \ref{compathe}.

\subsection{Triviality}
Let $A = (A_1,\ldots, A_n)$ be a commuting tuple of linear operators on a
vector space $E$. Let $1 \leq i < j \leq n$ be two numbers between $1$ and
$n$. We will assume that the commuting tuple obtained from $A$ by removing
both of the operators $A_i$ and $A_j$ is Fredholm. Thus,
\[
(ij)(A) = (A_1,\ldots,\wih A_i, \ldots, \wih A_j, \ldots, A_n)
\]
is a commuting Fredholm tuple. It then follows from Theorem \ref{indtriv} that
the commuting tuples
\[
\arr{ccc}{
i(A) = (A_1,\ldots, \wih A_i ,\ldots, A_n) & \T{and} &
j(A) = (A_1,\ldots, \wih A_j ,\ldots, A_n)
}
\]
are Fredholm. We can thus make sense of the joint torsion transition number
$\tau_{i,j}(A) \in \ff^*$. The aim of this section is to prove that
$\tau_{i,j}(A) = 1$.

We define the odd homotopy exact bitriangle $X$ by the diagram of odd
anti-commuting chain maps
\begin{equation}\label{eq:bitriv}
\begin{CD}
K\big((ij)(A)\big) @>{A_i}>> K\big((ij)(A)\big)[1] @>{\io_i}>>
K\big(j(A)\big) @>{\io_i^* \ep_i^*}>> K\big((ij)(A)\big) \\
@V{A_j}VV @V{-A_j}VV @V{A_j}VV @V{A_j}VV \\
K\big((ij)(A)\big)[1] @>{A_i}>> K\big((ij)(A)\big) @>{\io_i}>>
K\big(j(A)\big)[1] @>{-\io_i^* \ep_i^*}>> K\big((ij)(A)\big)[1] \\
@V{\io_{j-1}}VV @V{-\io_{j-1}}VV @V{\io_j}VV @V{\io_{j-1}}VV \\
K\big(i(A)\big) @>{A_i}>> K\big(i(A)\big)[1] @>{\io_i}>>
K(A) @>{\io_i^* \ep_i^*}>> K\big(i(A)\big) \\
@V{\io_{j-1}^* \ep_{j-1}^*}VV @V{-\io_{j-1}^* \ep_{j-1}^*}VV 
@V{\io_j^* \ep_j^*}VV @V{\io_{j-1}^* \ep_{j-1}^*}VV \\
K\big((ij)(A)\big) @>{A_i}>> K\big((ij)(A)\big)[1] @>{\io_i}>>
K\big(j(A)\big) @>{\io_i^* \ep_i^*}>> K\big((ij)(A)\big)
\end{CD}
\end{equation}
and the diagram of odd homotopies
\[
\begin{CD}
K\big((ij)(A)\big) @<<{0}< K\big((ij)(A)\big)[1] @<<{\io_i^*}<
K\big(j(A)\big) @<<{\ep_i \io_i}< K\big((ij)(A)\big) \\
@AA{0}A @AA{0}A @AA{0}A @AA{0}A \\
K\big((ij)(A)\big)[1] @<<{0}< K\big((ij)(A)\big) @<<{\io_i^*}<
K\big(j(A)\big)[1] @<<{-\ep_i \io_i}< K\big((ij)(A)\big)[1] \\
@AA{\io_{j-1}^*}A @AA{-\io_{j-1}^*}A @AA{\io_j^*}A @AA{\io_{j-1}^*}A \\
K\big(i(A)\big) @<<{0}< K\big(i(A)\big)[1] @<<{\io_i^*}<
K(A) @<<{\ep_i \io_i}< K\big(i(A)\big) \\
@AA{\ep_{j-1} \io_{j-1}}A  @AA{-\ep_{j-1} \io_{j-1}}A @AA{\ep_j \io_j}A 
@AA{\ep_{j-1} \io_{j-1}}A \\
K\big((ij)(A)\big) @<<{0}< K\big((ij)(A)\big)[1] @<<{\io_i^*}<
K\big(j(A)\big) @<<{\ep_i \io_i}< K\big((ij)(A)\big)
\end{CD}
\]
We leave it to the reader to verify the appropriate identities, see Definition
\ref{defhomoex}. Notice in this respect that the rows and columns are homotopy
exact by Theorem \ref{minusex}. The remaining anti-commutativity relations are
in fact satisfied at the level of chain complexes. We will use the odd
homotopy exact bitriangle $X$ to obtain the triviality result. Remark that the
homology group $H(X)$ is finite dimensional by assumption.

\begin{prop}
The joint torsion transition number $\tau_{i,j}(A) = 1$ is trivial.
\end{prop}
\begin{proof}
We need to prove that the torsion isomorphisms 
\[
(-1)^{\mu_i(A)}T_i(A) \, \T{ and }\, (-1)^{\mu_j(A)}T_j(A) : \deht{+}{A} \to
\deht{-}{A}
\]
agree. Here the non-negative numbers $\mu_i(A) \T{ and }\mu_j(A) \in \nn \cup
\{0\}$ are defined in \eqref{eq:signdef}. We will do this by applying the
comparison Corollary \ref{compcor}. The identifications of determinants which
appear are carried out according to the sign convention stated in the
beginning of Section \ref{jtorsion}.

We start by looking at the tensor product of torsion isomorphisms of the rows
in $X$. We then have the identifications
\[
\begin{split}
T(X_{1*}) \ot T(X_{2*}) \ot T(X_{3*})
& = (-1)^{\dht{-}{j(A)}} T(X^{j(A)}_i) \ot T(X^{j(A)}_i)^{-1} \ot T(X^A_i) \\
& \cong (-1)^{\dht{-}{j(A)}} T(X^A_i) \\
& \cong (-1)^{\dht{-}{j(A)}} T_i(A)
\end{split}
\]
for this tensor product. Remark that the sign on the last odd chain map of the
second row is responsible for the sign
\[
(-1)^{\dht{-}{j(A)} + \dht{-}{(ij)(A)} + \dht{-}{(ij)(A)}} = (-1)^{\dht{-}{j(A)}}
\]

Likewise, we look at the tensor product of torsion isomorphisms of the columns
in $X$. We then have the identifications
\[
\begin{split}
T(X_{*1}) \ot T(X_{*2}) \ot T(X_{*3})
& = (-1)^{\T{Ind}\big((ij)(A) \big) + \dht{-}{i(A)}}
T(X^{i(A)}_{j-1}) \ot T(X^{i(A)}_{j-1})^{-1} \ot T(X^A_j) \\
& \cong (-1)^{\T{Ind}\big((ij)(A) \big) + \dht{-}{i(A)}} T(X^A_j) \\
& \cong (-1)^{\T{Ind}\big((ij)(A) \big) + \dht{-}{i(A)}} T_j(A)
\end{split}
\]
for this tensor product. 

Finally, we look at the sign of the permutation
\[
\te : H_+(X_v) \op H_-(X_v) \to H_+(X_h) \op H_-(X_h)
\]
of homology groups. See Corollary \ref{compcor}. It can be verified that this
sign is given by 
\[
\T{sgn}(\te) = (-1)^{\T{Ind}\big((ij)(A)\big)}
\]

From the comparison of vertical and horizontal torsion isomorphisms we then
get the identity
\[
(-1)^{\dht{-}{i(A)}}T_j(A) = (-1)^{\dht{-}{j(A)}} T_i(A)
\]
The result of the lemma now follows by the triviality of the Fredholm indices,
\[
\T{Ind}(i(A)) = 0 = \T{Ind}(j(A))
\]
See Theorem \ref{indtriv}. Indeed, this vanishing result implies that
\[
\arr{ccc}{
(-1)^{\mu_k(A)} = (-1)^{\dht{+}{k(A)} \cd \dht{-}{k(A)}} =
(-1)^{\dht{-}{k(A)}} & & k = i,j
}
\]
\end{proof}



\subsection{Multiplicativity}
Let $A = (A_1,\ldots,A_n)$ and $B = (B_1,A_2,\ldots,A_n)$ be two $n$-tuples of
commuting linear operators on a vector space $E$. The two tuples only differ
in the first coordinate. Remark that we do not assume that the linear
operators $A_1$ and $B_1$ commute. Let us also fix two numbers $i,j \in
\{1,\ldots,n\}$.

We define the product of $A$ and $B$ as the $n$-tuple of commuting linear
operators
\[
A \cd B := (A_1 \cd B_1, A_2,\ldots,A_n)
\]
The goal of this section is to prove a multiplicativity relation for the joint
torsion transition numbers. To be more precise, we will prove that we have the
identity
\[
\tau_{i,j}(A \cd B) = \tau_{i,j}(A) \cd \tau_{i,j}(B)
\]
when the joint torsion transition numbers make sense. Notice that the slightly
more general result stated in Theorem \ref{mul} follows by the symmetry
property of the joint torsion transition numbers, see Theorem
\ref{jointsymm}.

An important tool is the relation of Corollary \ref{compcor} between vertical
and horizontal torsion isomorphisms. In order to apply this result we need to
establish the relevant homotopy exact bitriangles. This will be accomplished
in the next lemmas.

To begin with, we will establish a link between the Koszul complexes $K(B)$,
$K(A \cd B)$ and $K(A)$. The link is given by the following odd triangle of
$\zz_2$-graded chain complexes
\[
\begin{CD}
M(A,B) \, \, : \, \, K(B) @>{\nu(A_1)}>> K(A \cd B)[1] @>{\mu(B_1)}>> K(A)
@>{\ep_1^*}>> K(B)
\end{CD}
\]
Here the odd chain maps 
\[
\arr{ccc}{
\nu(A_1) : K(A) \to K(A \cd B)[1] & \T{and} & \mu(B_1) : K(A \cd B)[1] \to K(B)
}
\]
are given by
\[
\arr{ccc}{
\nu(A_1) := \ep_1 \ep_1^* + A_1 \ep_1^* \ep_1 & \T{and} &
\mu(B_1) := B_1 \ep_1 \ep_1^* + \ep_1^* \ep_1
}
\]
We would like to show that $M(A,B)$ is homotopy exact. To this end we define
the odd homotopy
\[
\begin{CD}
M^\da(A,B) \, \, : \, \,
K(B) @<<{\ep_1 \ep_1^*}< K(A \cd B)[1] @<<{\ep_1^* \ep_1}< K(A)
@<<{\ep_1}< K(B)
\end{CD}
\]

\begin{lemma}\label{multtri}
The odd triangle of $\zz_2$-graded chain complexes $M(A,B)$ is homotopy
exact.
\end{lemma}
\begin{proof}
We should start by checking that the composition of any two succesive odd
chain maps in the odd triangle $M(A,B)$ is chain homotopic to zero with
homotopies given by $M^\da(A,B)$. See Definition \ref{homoex}. The
verifications are all straight forward. However, for the convenience of the
reader, we will present the most complicated computation. We have that
\[
\mu(B_1) \ci \nu(A_1) = B_1 \ep_1 \ep_1^* + A_1 \ep_1^* \ep_1
\]
On the other hand we have that
\[
\begin{split}
d^A \ep_1 + \ep_1 d^B 
& = A_1 \ep_1^* \ep_1 + B_1 \ep_1 \ep_1^*
+ \sum_{i=2}^n A_i \ep_i^* \ep_1  + \sum_{i=2}^n A_i \ep_1 \ep_i^* \\
& = A_1 \ep_1^* \ep_1 + B_1 \ep_1 \ep_1^*
\end{split}
\]
This proves that $\mu(B_1) \ci \nu(A_1) : K(B) \to K(A)$ is chain homotopic to
zero with homotopy given by $\ep_1 : K(B) \to K(A)$.

We should then check the "homotopy decomposition" condition of Definition
\ref{homoex}. Again, everything follows by straight forward computations. In
fact, the desired identities are all valid at the level of chain complexes.
\end{proof}

Remark that a combination of Lemma \ref{multtri} and Lemma \ref{exhomol}
yields a proof of the additivity property for Fredholm tuple indices. See
Theorem \ref{multlin}.

Let $m \in \{2,\ldots,n\}$ we will now define a homotopy exact bitriangle
$X(m)$ of $\zz_2$-graded chain complexes. It is given by the following
anti-commuting diagram of odd chain maps
\[
\begin{CD}
K\big(m(B) \big) @>{A_m}>> K\big( m(B) \big)[1] @>{\io_m}>> K(B)
@>{\io_m^*\ep_m^*}>> K\big( m(B) \big) \\
@V{\nu(A_1)}VV @V{-\nu(A_1)}VV @V{\nu(A_1)}VV @V{\nu(A_1)}VV \\
K\big(m(A \cd B) \big)[1] @>{A_m}>> K\big( m(A \cd B) \big) @>{\io_m}>> K(A
\cd B)[1]
@>{- \io_m^*\ep_m^*}>> K\big( m(A \cd B) \big)[1] \\
@V{\mu(B_1)}VV @V{-\mu(B_1)}VV @V{\mu(B_1)}VV @V{\mu(B_1)}VV \\
K\big(m(A) \big) @>{A_m}>> K\big( m(A) \big)[1] @>{\io_m}>> K(A)
@>{\io_m^*\ep_m^*}>> K\big( m(A) \big) \\
@V{\ep_1^*}VV @V{-\ep_1^*}VV @V{\ep_1^*}VV @V{\ep_1^*}VV \\
K\big(m(B) \big) @>{A_m}>> K\big( m(B) \big)[1] @>{\io_m}>> K(B)
@>{\io_m^*\ep_m^*}>> K\big( m(B) \big)
\end{CD}
\]
and the diagram of odd homotopies
\[
\begin{CD}
K\big(m(B) \big) @<<{0}< K\big( m(B) \big)[1] @<<{\io_m^*}< K(B)
@<<{\ep_m \io_m}< K\big( m(B) \big) \\
@AA{\ep_1 \ep_1^*}A @AA{-\ep_1 \ep_1^*}A @AA{\ep_1 \ep_1^*}A @AA{\ep_1 \ep_1^*}A \\
K\big(m(A \cd B) \big)[1] @<<{0}< K\big( m(A \cd B) \big) @<<{\io_m^*}< K(A
\cd B)[1]
@<<{- \ep_m \io_m}< K\big( m(A \cd B) \big)[1] \\
@AA{\ep_1^* \ep_1}A @AA{-\ep_1^* \ep_1}A @AA{\ep_1^* \ep_1}A @AA{\ep_1^* \ep_1}A  \\
K\big(m(A) \big) @<<{0}< K\big( m(A) \big)[1] @<<{\io_m^*}< K(A)
@<<{\ep_m \io_m}< K\big( m(A) \big) \\
@AA{\ep_1}A @AA{-\ep_1}A @AA{\ep_1}A @AA{\ep_1}A \\
K\big(m(B) \big) @<<{0}< K\big( m(B) \big)[1] @<<{\io_m^*}<
K(B)
@<<{\ep_m \io_m}< K\big( m(B) \big)
\end{CD}
\]
Notice that the vertical triangles are homotopy exact by Lemma \ref{multtri}
and that the horizontal triangles are homotopy exact by Lemma
\ref{minusex}. This reduces the proof of the homotopy exactness of the
bitriangle $X(m)$ to a vertification of the anti-commutativity of several
squares. All the calculations involved are however next to trivial and they
will therefore not be considered at this place. Remark that the desired
identities are actually satisfied at the level of chain complexes.

Suppose that two of the commuting tuples $m(A), m(B)$ and $m(A \cd B)$ are
Fredholm. It then follows from Lemma \ref{multlin} that the third commuting
tuple is Fredholm as well. Furthermore, we see that the homology group $H\big(
X(m) \big)$ is finite dimensional. We will use the homotopy exact bitriangle
$X(m)$ to compare the torsion isomorphisms
\[
\arr{ccc}{
T_m(A) : \deht{+}{A} \to \deht{-}{A} & \T{and} & T_m(B) : \deht{+}{B} \to
\deht{-}{B}
}
\]
with the torsion isomorphism
\[
T_m(A \cd B) : \deht{+}{A \cd B} \to \deht{-}{A \cd B}
\]
This is carried out in the next lemma.

\begin{lemma}\label{casenotone}
We have a canonical identification of torsion isomorphisms
\[
\begin{split}
T\big( M(A,B) \big)
& \cong
(-1)^{\dhe{-}{A \cd B} + \mu_m(B) + \mu_m(A \cd B) + \mu_m(A)} \cd T_m(B) \ot
T_m(A \cd B)^{-1} \ot T_m(A) : \\
& \qq \dehe{+}{M(A,B)} \to \dehe{-}{M(A,B)}
\end{split}
\]
Here the exponents $\mu_m(B), \mu_m(A \cd B) \T{ and } \mu_m(A) \in \nn \cup
\{0\}$ are defined in \eqref{eq:signdef}.
\end{lemma}
\begin{proof}
We will use the sign convention appearing in the beginning of Section
\ref{jtorsion}.

We start by looking at the tensor product of torsion isomorphisms associated
with the rows of the homotopy exact bitriangle $X(m)$. We then have the
identifications
\[
\begin{split}
& T\big(X(m)_{1*}\big) \ot T\big( X(m)_{2*} \big) \ot T\big( X(m)_{3*}\big) \\
& \q = 
(-1)^{\dht{-}{A \cd B}} T(X^B_m) \ot T(X^{A \cd B}_m)^{-1} \ot T(X^A_m) \\
& \q \cong (-1)^{\dht{-}{A \cd B}}
T_m(B) \ot T_m(A \cd B)^{-1} \ot T_m(A)
\end{split}
\]
of isomorphisms. Remark that the sign on the last odd chain map in the second
row is responsible for the sign
\[
(-1)^{\dht{-}{A \cd B} + 2\cd \dht{-}{m(A \cd B)}} = (-1)^{\dht{-}{A \cd B}}
\]

Next, we look at the tensor product of torsion isomorphisms associated with
the columns of the homotopy exact bitriangle $X(m)$. We then have the
identifications
\[
\begin{split}
& T\big(X(m)_{*1}\big) \ot T\big( X(m)_{*2} \big) \ot T\big( X(m)_{*3}\big) \\
& \q = (-1)^{\dht{-}{X(m)_{*2}}} \cd
T\big( M( m(A),m(B)) \big) \ot T\big(M( m(A),m(B)) \big)^{-1} \ot T(M(A,B)) \\
& \q \cong (-1)^{\dht{-}{X(m)_{*2}}}T(M(A,B))
\end{split}
\]
of isomorphisms.

We will now examinate the sign of the permutation
\[
\te : H_+\big(X(m)_v\big) \op H_-\big(X(m)_v\big)
\to H_+\big(X(m)_h\big) \op H_-\big(X(m)_h\big)
\]
of vertical and horizontal homology groups. It can be verified that this sign
is given by
\[
\begin{split}
\T{sgn}(\te) 
& = (-1)^{\mu_m(B) + \mu_m(A) + \mu_m(A \cd B) + \dht{-}{X(m)_{*2}}}
\end{split}
\]
The result of the lemma then follows by an application of the comparison
Corollary \ref{compcor}.
\end{proof}

We will now treat the case where $m = 1$.

To this end, we define the odd homotopy exact bitriangle $X(1)$ by the diagram
\[
\begin{CD}
K\big(1(B) \big) @>{B_1}>> K\big( 1(B) \big)[1] @>{\io_1}>> K(B)
@>{\io_1^* \ep_1^*}>> K\big( 1(B) \big) \\
@V{1}VV @V{-A_1}VV @V{\nu(A_1)}VV @V{1}VV \\
K\big(1(A \cd B) \big)[1] @>{A_1 \cd B_1}>> K\big(1(A \cd B)\big) @>{\io_1}>>
K(A\cd B)[1] @>{-\io_1^* \ep_1^*}>> K\big( 1(A \cd B)\big)[1] \\
@V{0}VV @V{-\io_1}VV @V{\mu(B_1)}VV @V{0}VV \\
0 @>{0}>> K(A)[1] @>{1}>> K(A)
@>{0}>> 0 \\
@V{0}VV @V{-\io_1^* \ep_1^*}VV @V{\ep_1^*}VV @V{0}VV \\
K\big(1(B) \big) @>{B_1}>> K\big( 1(B) \big)[1] @>{\io_1}>> K(B)
@>{\io_1^* \ep_1^*}>> K\big( 1(B) \big)
\end{CD}
\]
of odd chain maps. This diagram is anti-commutative except for the square
\[
\begin{CD}
K\big(1(A \cd B) \big)[1] @>{A_1 \cd B_1}>> K\big(1(A \cd B)\big) \\
@V{0}VV @V{-\io_1}VV \\
0 @>{0}>> K(A)[1]
\end{CD}
\]
The even chain map
\[
- \io_1 \ci (A_1 B_1) : K\big(1(A \cd B) \big)[1] \to K(A)[1]
\]
is chain homotopic to zero through the homotopy
\[
s := \ep_1 \io_1 B_1 :  K\big(1(A \cd B) \big)[1] \to K(A)[1]
\]

The diagram of odd homotopies for $X(1)$ is given by
\[
\begin{CD}
K\big(1(B) \big) @<<{0}< K\big( 1(B) \big)[1] @<<{\io_1^*}< K(B)
@<{\ep_1 \io_1}<< K\big( 1(B) \big) \\
@AA{1}A @AA{0}A @AA{\ep_1 \ep_1^*}A @AA{1}A \\
K\big(1(A \cd B) \big)[1] @<<{0}< K\big(1(A \cd B)\big) @<<{\io_1^*}<
K(A\cd B)[1] @<<{-\ep_1 \io_1}< K\big( 1(A \cd B)\big)[1] \\
@AA{0}A @AA{-\io_1^*}A @AA{\ep_1^* \ep_1}A @AA{0}A \\
0 @<<{0}< K(A)[1] @<<{1}< K(A)
@<<{0}< 0 \\
@AA{0}A @AA{-\ep_1 \io_1}A @AA{\ep_1}A @AA{0}A \\
K\big(1(B) \big) @<<{0}< K\big( 1(B) \big)[1] @<<{\io_1^*}< K(B)
@<<{\ep_1 \io_1}< K\big( 1(B) \big)
\end{CD}
\]
In order to prove that $X(1)$ is a homotopy exact bitriangle we need to
consider several identities. See Definition \ref{defhomoex}. First of all we
notice that the homotopy exactness of the rows and columns follows from Lemma
\ref{multtri} and Theorem \ref{minusex}. We should then prove that the linear
maps
\[
vs + sv + ht + th \, \T{ and } \, hs + sh + vr + rv : 
H\big( X(1)\big) \to H\big( X(1) \big)
\]
are trivial. It turns out that these maps are already trivial at the level of
chain complexes. As an example, let us verify the identity
\[
hs + sh + vr + rv = 0 : K(A\cd B)[1] \to K(A)[1]
\]
On this component we have that
\[
\begin{split}
hs + sh + vr + rv 
& = - \ep_1 \io_1 B_1 \io_1^* \ep_1^*  - \io_1 \io_1^* + \mu(B_1) \\
& = - B_1 \ep_1 \ep_1^* - \ep_1^* \ep_1 + \mu(B_1) \\
& = 0
\end{split}
\]
The rest of the vanishing results follows by similar straight forward
calculations. Let us suppose that the commuting tuple $1(A) = 1(B) = 1(A\cd
B)$ is Fredholm. In particular, we get that the homology group $H\big( X(1)
\big)$ is finite dimensional. We will use the homotopy exact bitriangle $X(1)$
to compare the torsion isomorphisms
\[
\arr{ccc}{
T_1(A) : \deht{+}{A} \to \deht{-}{A} & \T{and} & T_1(B) : \deht{+}{B} \to
\deht{-}{B}
}
\]
with the torsion isomorphism
\[
T_1(A \cd B) : \deht{+}{A \cd B} \to \deht{-}{A \cd B}
\]
This is carried out in the next lemma.

\begin{lemma}\label{caseone}
We have a canonical identification of torsion isomorphisms
\[
\begin{split}
& T\big(M(A,B)\big)
\cong (-1)^{\dhe{-}{A \cd B} + \mu_1(A)} 
T_1(B) \ot T_1(A \cd B)^{-1} \ot T_1(A) : \\
& \qq \dehe{+}{M(A,B)} \to \dehe{-}{M(A,B)}
\end{split}
\]
\end{lemma}
\begin{proof}
We will use isomorphisms similar to the sign convention stated in the
beginning of Section \ref{jtorsion}.

Let us look at the tensor product of vertical torsion isomorphisms associated
with the homotopy exact bitriangle $X(1)$. We then have the identifications
\[
\begin{split}
& T\big(X(1)_{*1}\big) \ot T\big(X(1)_{*2}\big) \ot T\big(X(1)_{*3}\big) \\
& \q = (-1)^{\T{Ind}\big( 1(A)\big) + \dht{-}{A}} 
\cd T\big( X(1)_{*1} \big) \ot T(X^A_1)^{-1} \ot T\big(M(A,B) \big) \\
& \q \cong
(-1)^{\T{Ind}\big( 1(A) \big) + \dht{-}{A} + \mu_1(A)}
\cd T_1(A)^{-1} \ot T\big( M(A,B) \big)
\end{split}
\]
for this tensor product.

Next, we take a look at the tensor product of horizontal torsion isomorphisms associated
with the homotopy exact bitriangle $X(1)$. We then have the identifications
\[
\begin{split}
& T\big(X(1)_{1*}\big) \ot T\big(X(1)_{2*}\big) \ot T\big(X(1)_{3*}\big) \\
& \q = (-1)^{\dht{-}{A \cd B}} \cd 
T(X^B_1) \ot T(X^{A \cd B}_1)^{-1} \ot T\big( X(1)_{3*}\big) \\
& \q \cong (-1)^{\dht{-}{A \cd B} + \dht{-}{A}} T_1(B) \ot T_1(A \cd B)^{-1}
\end{split}
\]
for this tensor product. Here we have used that $\T{Ind}(A) = 0$. Indeed, the
torsion isomorphism $T\big( X(1)_{3*}\big)$ is really given by the sign
\[
T\big( X(1)_{3*} \big) \cong (-1)^{\dht{-}{A} \cd \dht{+}{A}} = (-1)^{\dht{-}{A}}
\]

Finally, we consider the sign of the permutation
\[
\te : H_+\big(X(1)_v\big) \op H_-\big(X(1)_v\big)
\to H_+\big(X(1)_h\big) \op H_-\big(X(1)_h\big)
\]
of homology groups. It can be checked that this sign is given by
\[
\T{sgn}(\te) = (-1)^{\T{Ind}\big( 1(A)\big)}
\]
The result of the lemma now follows by an application of the comparison
Corollary \ref{compcor}.
\end{proof}

We are now ready to prove the main result of this section: The
multiplicativity of the joint torsion transition numbers.

\begin{prop}
Suppose that the joint torsion transition numbers in position $(i,j)$ are
well-defined for two of the three commuting tuples $A$, $B$ and $A \cd
B$. Then the joint torsion transition number is well-defined for the third
commuting tuple and is linked to the two others by the multiplicativity
relation
\[
\tau_{i,j}(A \cd B) = \tau_{i,j}(A) \cd \tau_{i,j}(B)
\]
\end{prop}
\begin{proof}
The result follows from the definition of the joint torsion transition numbers
and by an application of Lemma \ref{casenotone} and Lemma \ref{caseone}.
\end{proof}

\end{document}